\documentclass[reqno,centertags, 11pt]{amsart}
\usepackage{comment,graphicx,url}
\usepackage{color}

\definecolor{purple}{rgb}{0.65, 0, 0.9}
\definecolor{orange}{rgb}{1,.5,0}
\definecolor{gray}{rgb}{0.7,.7,0.7}
\usepackage{enumitem}
\usepackage[]{hyperref}
\usepackage{mathtools}
\usepackage[normalem]{ulem}
\usepackage{esint}

\usepackage{booktabs} 
\usepackage{array} 
\usepackage{verbatim} 
\usepackage{subfig} 
\usepackage{mathrsfs}
\usepackage{amssymb}
\usepackage{amsthm}
\usepackage{amsmath,amsfonts,amssymb,esint,hyperref}
\usepackage[noabbrev, capitalize]{cleveref}
\usepackage{graphics,color}
\usepackage{enumitem}
\usepackage{mathtools,centernot}
\usepackage{cases}
\usepackage{amsrefs}
\usepackage{bbm}
\usepackage{xfrac}
\usepackage{xcolor}

\usepackage{graphics}
\usepackage{amssymb,amsmath,amsfonts}
-\marginparwidth \oddsidemargin -9mm \evensidemargin \oddsidemargin
\usepackage{latexsym}
\advance\hoffset by 5mm

\def\@abssec#1{\vspace{.1in}\footnotesize \parindent .2in
{\bf #1. }\ignorespaces}
\newtheorem{theorem}{Theorem}[section]

\newtheorem{lemma}[theorem]{Lemma}
\newtheorem{proposition}[theorem]{Proposition}
\newtheorem{corollary}[theorem]{Corollary}

\newtheorem{remark}[theorem]{Remark}

\newcommand*{\rom}[1]{\expandafter\@slowromancap\romannumeralq #1@}

\newcommand{\be}{\mathbf e}

\allowdisplaybreaks \numberwithin{equation}{section}

\renewcommand{\be}{\begin{equation}}
\newcommand{\ee}{\end{equation}}

\newcommand{\colvec}[1]{\begin{pmatrix} #1 \end{pmatrix}}

\begin{document}

\title[Sharp stability for IPM]{Sharp asymptotic stability of the incompressible porous media equation}

\author{Roberta Bianchini}
\address{Istituto per le Applicazioni del Calcolo ''Mauro Picone'', 00185, Rome, Italy}
\email{roberta.bianchini@cnr.it}

\author{Min Jun Jo}
\address{Department of Mathematics, Duke University, Durham, NC 27708 USA}
\email{minjun.jo@duke.edu}

\author{Jaemin Park}
\address{Department of Mathematics, Yonsei University, 03722, Seoul, South Korea}
\email{jpark776@yonsei.ac.kr}

\author{Shan Wang}
\address{Istituto per le Applicazioni del Calcolo ''Mauro Picone'', 00185, Rome, Italy}
\email{wangshan0624@gmail.com}

\subjclass[2020]{76S05 - 35Q35 -34D05  - 76B03}
\keywords{Asymptotic stability - incompressible fluids - porous media equation}
\begin{abstract}
In this paper, we prove the asymptotic stability of the incompressible porous media (IPM) equation near a stable stratified density, for initial perturbations in the Sobolev space \( H^k \) with any \( 2<k \in\mathbb{R} \). While it is known that such a steady state is unstable in \( H^2 \), our result establishes a sharp stability threshold in higher-order Sobolev spaces.

The key ingredients of our proof are twofold.  First, we extract long-time convergence from the decay of a potential energy functional{—despite its non-coercive nature—thereby revealing a variational structure underlying the dynamics.} 
Second, we derive refined commutator estimates to control the evolution of higher Sobolev norms throughout the full range of $k>2$.\color{black}

\end{abstract}

\maketitle
\setcounter{tocdepth}{1}

\tableofcontents

\setcounter{tocdepth}{1}
\section{Introduction}
\subsection{Incompressible porous media equation}
 In this paper, we investigate the asymptotic stability of the incompressible porous media (IPM) equation in the periodic strip $\Omega:=\mathbb{T}\times\mathbb{R}$:
 \begin{align}\label{IPM}
 \begin{cases}
 \rho_t +u\cdot\nabla{\rho} = 0&\text{ in $\Omega$,}\\
 \rho(0,x):=\rho_0(x),
 \end{cases}
 \end{align}
 where the velocity $u$ is uniquely determined by
 \begin{align}\label{Darcy_law}
  u=-\nabla p -\colvec{0 \\ \rho},\quad  \nabla\cdot u=0\quad \text{ in $\Omega$ and }\quad \lim_{|x|\to\infty}|u(x)|=0.
 \end{align}
  As usual, we denote the stream function by $\Psi$ such that $u=\nabla^\perp\Psi$. From the relation between the velocity and the advected scalar $\rho$ in \eqref{Darcy_law}, it follows that $\Psi$ is uniquely determined (up to a constant) as a solution to the following equation:
 \begin{align}\label{stream_IPM}
 \begin{cases}
  -\Delta \Psi = \partial_1\rho, & \text{ in $\Omega$},\\
  \lim_{|x|\to \infty}\Psi(x)= 0.
  \end{cases}
 \end{align}

 Let $\rho_s(x)$ be a stratified density, that is, $\rho_s(x)=\rho_s(x_2)$ is independent of the variable $x_1$. Formally speaking, such a density function induces zero velocity field, implying that it is a steady solution to the IPM equation. The goal of this paper is to study the asymptotic stability around stable steady states satisfying the following properties:  
  \begin{align}\label{density_property}
 \inf_{x_2\in\mathbb{R}}\left(- \partial_2\rho_s(x_2)\right)>0,\quad \rVert \partial_2\rho_s\rVert_{C^{k+1}(\Omega)}<\infty, \text{ for some $k>2$.}
  \end{align}
Clearly  a typical example can be  $\rho_s(x_2)=-x_2$.

In recent years, there has been active research on the well-posedness of the IPM equation. In \cite{MR2337005}, it was proved that the IPM equation is locally well-posed when the density belongs to $C^{1,\alpha}$ for $\alpha > 0$ or to $H^k$ for $k > 2$. This result relies on the fact that the velocity field is Lipschitz, ensuring the associated transport equation is well-defined. However, the question of global well-posedness for general initial data versus the formation of finite-time singularities remains open.

In one direction, the authors in \cite{MR4527834} established long-time growth of Sobolev norms, showing that for fairly general initial data, the norms may grow unbounded as time tends to infinity—assuming a global-in-time solution exists, even when the initial data is smooth. In particular, one of their results shows that even small perturbations of a stable steady state may lead to instability if the perturbation lies in $H^{2-\epsilon}$ for any $\epsilon > 0$. The space $H^2$ is critical for the IPM equation, and strong instability and ill-posedness were established by the first author of the present paper and her collaborators in \cite{bianchini2024non}. Remarkably, C\'ordoba--Mart{\'\i}nez-Zoroa \cite{cordoba2024finite} constructed a finite-time singularity for the IPM equation starting from smooth initial data and smooth forcing. We also mention the work \cite{kiselev2024finite}, in which the authors derived a 1D model of the IPM equation and proved finite-time singularity formation in the absence of forcing.

On the other hand, it has been shown that such (negative) behaviors do not occur near stable steady states when perturbations are taken in strong Sobolev spaces. To the best of our knowledge, this was first established by Elgindi \cite{elgindi2017asymptotic} in both $\mathbb{R}^2$ and $\mathbb{T}^2$, where he showed that if the initial perturbation is in $H^{20}$, the solutions eventually converge to a stratified steady state in the long-time limit. A related question was studied in \cite{castro2019global}, where the authors posed the equation in a periodic channel and proved asymptotic stability of a stable steady state in $H^{10}$. Further progress in this direction was made in \cite{bianchini2024relaxation,jo2024quantitative}, where the required regularity for stability was substantially reduced to $H^{3+}$. The current best result appears in \cite{park2024stability}, where asymptotic stability is established in $H^3$. The goal of the present paper is to push these results to their limit, establishing asymptotic stability in the optimal regularity class.

Before presenting our main result, we briefly introduce a closely related model, known as the AHT model \cite{angenent2003minimizing,brenier2009optimal}. Given a bounded domain $D \subset \mathbb{R}^d$ and a prescribed deformation $y_0 : D \to \mathbb{R}^d$, the system is given by:
\begin{align} \label{AHT_1}
\begin{cases}
y_t + v \cdot \nabla y = 0, \\
Kv = -\nabla p - y, \\
y(0,x) = y_0(x),
\end{cases}
\end{align}
where $K$ is a dissipative operator, such as $K = I$ or $K = -\Delta$.
This model was introduced to study the unique rearrangement $y^*$ with \textcolor{black}{a} convex potential, as it arises in Brenier’s polar factorization \cite{brenier1991polar}. In particular, \cite{nguyen2020global} proved exponential convergence to the minimizer when the initial data is close to a stable one, i.e., when $D y_0(x)$ is strictly positive definite for every $x \in D$.
A connection between the IPM and the AHT model can be observed by choosing $K = I$ and initial data of the form:
\[
y_0 = -\begin{pmatrix} 0 \\ \rho_0 \end{pmatrix}.
\]
With this choice, the vector transport system \eqref{AHT_1} becomes equivalent to the IPM equation. The key difference from the setting in \cite{nguyen2020global} is that, in this case, the IPM equation \eqref{IPM} corresponds to a degenerate scenario, since
\[
D y_0 = \begin{pmatrix} 0 & 0 \\ -\partial_1\rho_0 & -\partial_2 \rho_0 \end{pmatrix}
\]
is not strictly positive definite. While exponential convergence cannot be expected in this case, we will derive an algebraic convergence rate, which is expected to be sharp at least at the linear level.
 
\subsection{Main theorem}\label{main_theorem1}
 Before stating our theorem, we consider  potential candidates for a long-term limit of the solution, following the idea presented in \cite{dalibard2023long}. Let us consider $f\in C^1(\Omega)$ such that $\inf (-\partial_2f)>0$.  Since $f$ is monotonically decreasing in the vertical variable $x_2$,  the implicit function theorem tells us that for any $s\in\mathbb{R}$, we can find a unique level set $\left\{ x\in\Omega: f(x)=s\right\}$, which can be parametrized by the horizontal variable $x_1$. More precisely, we can find $\phi:\mathbb{T}\times \mathbb{R}\to\mathbb{R}$ such that
  \[
  f(x_1,\phi(x_1,s))=s,\text{ for all $x_1\in\mathbb{T}$.}
  \]
  Let us define $\phi_0:\mathbb{R}\to\mathbb{R}$ and $h:\mathbb{T}\times\mathbb{R}\to \mathbb{R}$ as
  \begin{equation}\label{def_average}
\phi_0(s):=\int_{\mathbb{T}}\phi(x_1,s)dx_1,\quad h(x_1,s):=\phi(x_1,s)-\phi_0(s).
\end{equation}
  By definition, $h$ satisfies 
  \[
  \int_{\mathbb{T}}h(x_1,s)dx_1=0.
  \]
  Using \eqref{def_average}, we can define the \textcolor{black}{\emph{measure-preserving stratification} $f^*$} of $f$ as
  \begin{align}\label{stratification_def}
  f^*(x_2):=\left(\phi_0\right)^{-1}(x_2).
  \end{align}
  An illustration of such decomposition can be seen in Figure~\ref{fig1}.
\begin{figure}[h!]
\hspace{0.3cm}\includegraphics[scale=1.2]{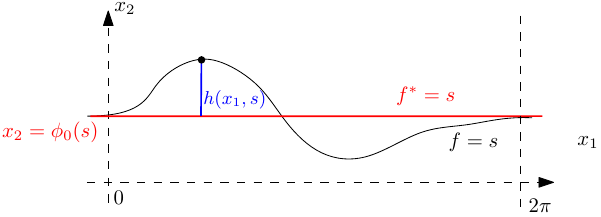}
\caption{An illustration of the level sets of $f$. Each level set $f=s$ is uniquely decomposed into $\phi_0$ and $h$ so that $\int_{\mathbb{R}}h(x_1,s)dx_1=0$ for each $s\in\mathbb{R}$.}
\label{fig1}
\end{figure}

 Now, let us consider $\rho(t)$ as a solution to the IPM equation, where $\rho_0$ is sufficiently close to $\rho_s$ in a topology stronger than $C^1$, and where $\rho_s$ satisfies \eqref{density_property}. Then, the measure-preserving stratification of $\rho_0$ is well-defined. Furthermore, if $\lim_{t \to \infty} \rho(t)$ exists and the limit can be obtained in the $C^1$-norm, then $\rho_0^*$ is the only possible stratified density under a transport equation governed by an incompressible vector field. Noting that $H^2$ is the borderline case that fails to embed into $C^1$ in a two-dimensional domain, we achieve its optimal asymptotic stability in our main theorem.
 \begin{theorem}\label{main_theorem}
Let $\rho_s$ be a stratified density satisfying \eqref{density_property}. Then $\rho_s$ is asymptotically stable for the IPM equation, \eqref{IPM}-\eqref{Darcy_law}, in $H^k(\Omega)$ for any $k\in\mathbb{R}$ such that $k>2$.

  More precisely, setting  \begin{align}\label{alpha_def}
\gamma:=\min\left\{ k-2, \  \inf_{x\in\Omega}\left(- \partial_2\rho_s(x_2)\right)\text{ and }\left(\rVert \partial_2\rho_s\rVert_{C^{k+1}}\right)^{-1}\right\}>0, 
  \end{align} 
  there exist $\varepsilon,\ C>0$ which depends only on $\gamma$ such that if $
\rVert \rho_0 -\rho_s\rVert_{H^k}\le \varepsilon,$ then there exists a unique solution $\rho(t)$ satisfying $\rho-\rho_s\in C([0,\infty) ; H^k)$ to the IPM equation, and the solution $\rho$ satisfies
\[
\rVert \rho(t) - \rho_s\rVert_{H^k}\le C\varepsilon,\text{ for all $t>0$.}
\]
Furthermore, the solution $\rho$ converges to $\rho^*_0$ as 
\begin{align}\label{convergence_1}
\rVert \rho(t)-\rho_0^*\rVert_{L^2}\le C\varepsilon t^{-k/2},\text{ for all $t>0$,}
\end{align}
where $\rho_0^*$ is the  measure-preserving stratification of the  initial density $\rho_0$ defined in \eqref{stratification_def}.
\end{theorem}

\begin{remark}
The convergence rate in \eqref{convergence_1} is sharp in the sense that for any $\varepsilon>0$, one can construct a solution $\rho^L(t)$ to the linearized equation around $\rho_s$ such that 
\[
\rVert \rho^L(t)-\rho_s\rVert_{L^2}\ge C t^{-k/2-\varepsilon} \text{ for all $t>1$}.
\]
A short proof is given in  Appendix~\ref{Sharpness_1}.
\end{remark}
\begin{remark}
Theorem~\ref{main_theorem} establishes a sharp regularity result in the sense that such $\rho_s$ is known to be unstable in $H^2$ (see \cite{MR4527834} and \cite{bianchini2024non}, where long-time instability in  $H^{2-\epsilon}$ and ill-posedness in $H^2$ have been proved respectively). Previously, the best known stability result \cite{park2024stability} was obtained in $H^3$ on the horizontally periodic strip, improving the regularity requirement in the preceding results \cite{bianchini2024relaxation,castro2019global,elgindi2017asymptotic,jo2024quantitative}.
\end{remark}

\begin{remark}
    It is very likely that our proof could be adapted to establish sharp asymptotic stability of the two-dimensional Boussinesq equations with velocity damping around any equilibrium of the form $(\rho_s, u_s) = ( \rho_s(x_2), 0)$, where $\rho_s(x_2)$ is a stably stratified density profile. Previous results in this direction include: \cite{bianchini2024relaxation}, which treats the equations posed on the whole space $\mathbb{R}^2$ with regularity requirement $H^{s}$ with $s>3$ (see also \cite{wan2019} for results under higher regularity assumptions), and \cite{castro2019}, which considers the case of a finite periodic channel with regularity requirement $H^{k}$ with $k>17$. 
\end{remark}

\subsection{Brief overview of the proof}  Previous results establishing asymptotic stability for the IPM equation (e.g., \cites{bianchini2024relaxation,castro2019global,elgindi2017asymptotic,jo2024quantitative}) rely on Fourier-analytic methods, which explicitly capture the underlying stability mechanism (see Subsubsection~\ref{subsection_linear}). However, such an approach faces several challenges. In particular, deriving key commutator estimates required for achieving optimal regularity of the initial perturbation can be technically difficult, primarily due to potential derivative losses that might occur in rough control of nonlinear terms. Moreover, within the Fourier-analytic framework, the precise long-time asymptotic behavior of the density—specifically, its quantitative convergence to the measure-preserving stratification of the initial data, as characterized in this paper—remains elusive.

The main novelty of our proof of Theorem~\ref{main_theorem} lies in the use of a potential energy functional. This reveals a variational structure of the problem more transparently, thereby providing complete description of the long-time limit of the solution. Another key ingredient toward establishing the optimal regularity requirement on the initial data is the incorporation of additional decay estimates for higher Sobolev norms. A similar strategy was previously adopted in \cite{park2024stability} where an improved regularity requirement, $k\geq 3$, was successfully obtained. In the present work, we build further upon the described strategy to establish asymptotic stability with the optimal regularity threshold. The optimality is accomplished here by sharply exploiting the gap between the subcritical exponent $k>2$ and the critical exponent $k=2$ to control the pivotal quantity $\|\nabla u_2\|_{L^\infty}$ in the aforementioned step of incorporating additional decay estimates. For the reader’s convenience, we briefly outline the structure of our proof in what follows.

\subsubsection{Decay rates in the linearized equation}\label{subsection_linear}
We begin by investigating the decay rates of the velocity field in the linearized IPM equation. For simplicity, we take a look at the linearized system around the simplest stable stratified steady state $\rho_s(x) := -x_2$, and consider initial data with zero horizontal average:
\[
\begin{cases}
\Theta_t = U_2,\\
U = \nabla^{\perp}(-\Delta)^{-1} \partial_1 \Theta,\\
\Theta(0,x) = \Theta_0(x) \quad \text{with } \int_{\mathbb{T}} \Theta_0(x_1, x_2) \, dx_1 = 0.
\end{cases}
\]
Using the identity $U_2 = (-\Delta)^{-1} \partial_{11} \Theta$, we can explicitly solve the linearized equation via Fourier transform (see \eqref{fourier_notation} for notation):
\[
\Theta(t,x_1,x_2) = \sum_{n \in \mathbb{Z} \setminus \{0\}} \int_{\mathbb{R}} \widehat{\Theta_0}_n(\xi) e^{-t \frac{n^2}{n^2 + \xi^2}} e^{i(n x_1 + \xi x_2)} \, d\xi.
\]
For $k > 2$ and $s_1, s_2 \in [0, k]$ with $s_1 + s_2 \le k$, a direct computation (see \eqref{d_denf1}) yields:
\[
\| |D_1|^{s_1} \Theta(t) \|_{\dot{H}^{s_2}}^2 = \sum_{n \neq 0} \int_{\mathbb{R}} |\widehat{\Theta_0}_n(\xi)|^2 (n^2 + \xi^2)^k W(t,n,k) \, d\xi,
\]
where
\[
W(t,n,k) := \frac{n^{2s_1}}{(n^2 + \xi^2)^{k - s_2}} e^{-\frac{2tn^2}{n^2 + \xi^2}}.
\]
Noting that $|n| \ge 1$ and that $\sup_{x > 0} x^\alpha e^{-x} < C_\alpha$ for any $\alpha > 0$, we obtain the bound:
\[
W(t,n,k) \le C_k t^{-(k - s_2)}.
\]
Consequently,
\[
\| |D_1|^{s_1} \Theta(t) \|_{H^{s_2}}^2 \le C(k, \| \Theta_0 \|_{H^k}) \, t^{-(k - s_2)}.
\]
Recalling $U = \nabla^\perp(-\Delta)^{-1} \partial_1 \Theta$, this implies
\begin{align}\label{velocity_linear_1}
\| \Theta(t) \|_{L^2}^2 \le C t^{-k}, \quad \| U(t) \|_{H^2}^2 \le C t^{-(k - 1)}, \quad \| U_2(t) \|_{H^2}^2 \le C t^{-k},
\end{align}
for some constant $C = C(k, \| \Theta_0 \|_{H^k}) > 0$.

Whether such decay rates persist at the nonlinear level is far from trivial. In our analysis, we derive an \emph{averaged} version of these estimates for nonlinear solutions (see Proposition~\ref{u_2high1}):
\begin{align}\label{usdd2es}
\| u(t) \|_{H^2}^2 = \| \rho_0 - \rho_s \|_{H^k}^2 O(t^{-(k - 1)}), \quad
\| u_2(t) \|_{H^2}^2 = \| \rho_0 - \rho_s \|_{H^k}^2 O(t^{-k}),
\end{align}
where we use the following (slightly non-standard) time-averaged notation:
\begin{align}\label{time_average_notation1}
f(t) = O(t^{-\alpha}) \quad \text{if} \quad \frac{1}{t/2} \int_{t/2}^t f(s) \, ds \le \frac{C}{(1 + t)^\alpha} \quad \text{for } t \gg 1 \text{ and some } C, \alpha > 0.
\end{align}

Moreover, the decay rate of $\Theta$ in \eqref{velocity_linear_1} coincides with the convergence rate in \eqref{convergence_1}.

\subsubsection{Nonlinear analysis}
The asymptotic stability analysis for a solution $\rho(t)$ to the (nonlinear) IPM equation is carried out under the following a priori assumption:
\begin{align}\label{assumption_main}
\| \rho(t)-\rho_s \|_{H^k}^2 + \int_0^T \| u(t) \|_{H^k}^2 \,dt < M\varepsilon^2 \ll 1, \quad \text{for all } t \in [0,T], \text{ for some } T, M > 0 \text{ and } 0 < \varepsilon \ll 1,
\end{align}
where the initial perturbation satisfies
\[ \| \rho_0 - \rho_s \|_{H^k} \le \varepsilon. \]

We aim to prove the following two statements in parallel:
\begin{itemize}
  \item[(A)] \textbf{Convergence to the energy minimizer $\rho_0^*$:}
  \[
  \| \rho(t) - \rho_0^* \|_{L^2}^2 \le \frac{C_k \varepsilon^2}{t^k}, \quad \text{as long as \eqref{assumption_main} holds.}
  \]
  \item[(B)] \textbf{Persistence of the bound in \eqref{assumption_main}} for all $T > 0$.
\end{itemize}
Verifying (A) and (B) together establishes Theorem~\ref{main_theorem}. We now outline the proofs separately.

\textbf{Proof scheme for (A).}
As previously discussed, our proof is based on the decay of the potential energy functional. Specifically, we consider:
\[
\mathcal{E}(\rho(t)) := \lim_{s \to \infty} \left( \int_{-s < \rho(t) < s} \rho(t,x) x_2 \, dx - \int_{-s < \rho_0^* < s} \rho^*(x) x_2 \, dx \right),
\]
where $\rho_0^*$ is the measure-preserving rearrangement of the initial data defined in \eqref{stratification_def}. An important property of $\mathcal{E}$ is that $\rho_0^*$ uniquely minimizes $\mathcal{E}$ among all measure-preserving rearrangements of $\rho_0$. Since $\rho(t)$ solves a transport equation with a divergence-free velocity field, $\mathcal{E}$ remains nonnegative. Under assumption~\eqref{assumption_main}, the energy functional satisfies both upper and lower bounds (see Proposition~\ref{propoos}):
\begin{align}\label{energy_as_a_norm}
\| \rho(t) - \rho_0^* \|_{L^2}^2 \lesssim \mathcal{E}(\rho(t)) \lesssim \| \rho(t) - \rho_0^* \|_{L^2}^2, \quad \mathcal{E}(\rho(0)) \lesssim \varepsilon^2.
\end{align}

As observed in \cite{elgindi2017asymptotic}, the functional $\mathcal{E}$ is monotone decreasing:
\begin{align}\label{energy_derivatives_1}
\frac{d}{dt} \mathcal{E}(\rho(t)) = - \| u(t) \|_{L^2}^2.
\end{align}
To prove decay of $\mathcal{E}(t)$, one would ideally derive a coercive inequality like:
\[
\| u(t) \|_{L^2}^2 \ge C \mathcal{E}(\rho(t)),
\]
which would yield exponential decay. However, such an inequality fails because the operator $\partial_1 (-\Delta)^{-1/2}$ is not positive definite: if it were true, then the energy bounds \eqref{energy_as_a_norm} and Darcy's law 
\[
u=\nabla^\perp(-\Delta)\partial_1(\rho-\rho_0^*)
\] would imply
\[
\| \rho - \rho_0^* \|_{L^2} \le C \| \partial_1 (-\Delta)^{-1/2} (\rho - \rho_0^*) \|_{L^2},
\]
which cannot be true in general. Instead, we exploit higher regularity and derive the following estimate:
\begin{align}\label{u_intro_briel}
\| u(t) \|_{L^2}^2 \ge C \mathcal{E}(\rho(t))^{\frac{k}{k-1}} \left( \| u(t) \|_{H^k}^2 \right)^{-\frac{1}{k-1}}.
\end{align}
The derivation of this inequality will be given in Proposition~\ref{propoos}, based on a decomposition of the level sets of $\rho$, illustrated in Figure~\ref{fig1}.

Meanwhile, assumption~\eqref{assumption_main} implies (using the notation \eqref{time_average_notation1}):
\begin{align}\label{uestimae1}
\| u(t) \|_{H^k}^2 = M \varepsilon^2 O(t^{-1}).
\end{align}
Thus (formally),
\[ \left( \| u(t) \|_{H^k}^2 \right)^{-\frac{1}{k-1}} = (M \varepsilon^2)^{-\frac{1}{k-1}} O(t^{\frac{1}{k-1}}). \]
Combining with \eqref{u_intro_briel}, we read the energy identity \eqref{energy_derivatives_1} as
\[
\frac{d}{dt} \mathcal{E} \lesssim -\mathcal{E}^{\frac{k}{k-1}} (M \varepsilon^2)^{-\frac{1}{k-1}} (1 + t)^{\frac{1}{k-1}},
\]
which implies
\[
\mathcal{E}(\rho(t)) \lesssim \frac{M \varepsilon^2}{(1 + t)^k}, \quad \text{as long as \eqref{assumption_main} holds}.
\]
Combined with \eqref{energy_as_a_norm}, this gives the desired convergence rate of $\rho(t)$ to $\rho_0^*$. A full derivation is given in Proposition~\ref{energy_analysis}.

 \textbf{Proof scheme for (B).}
To complete the argument, we must verify that assumption~\eqref{assumption_main} holds for all \( T > 0 \). This is achieved by estimating the evolution of the high Sobolev norm of \( \theta(t) := \rho(t) - \rho_0 \), specifically \( \| \theta(t) \|_{H^k}^2 \). Although standard techniques from transport equations apply, nontrivial technical difficulties arise due to the nonlinear structure of the IPM system. In particular, Proposition~\ref{energy_IPM_estimate} provides the differential inequality:
\[
\frac{d}{dt} \| \theta(t) \|_{H^k}^2 \le -C \| u(t) \|_{H^k}^2 + C \left( \| \nabla u_2(t) \|_{L^\infty} \| \theta(t) \|_{H^k}^2 + \| u(t) \|_{L^2}^2 \right).
\]
We emphasize that although the Lipschitz norm of the velocity vector naturally appears in such estimates for general transport equations, our bound involves only the Lipschitz norm of the vertical component \( u_2 \). Moreover, this estimate holds for the full range \( k > 2 \).

Integrating the inequality in time yields
\begin{align}\label{sia_1}
\sup_{t \in [0,T]} \| \theta(t) \|_{H^k}^2 + \int_0^T \| u(t) \|_{H^k}^2 \, dt &\lesssim \| \theta_0 \|_{H^k}^2 + \sup_{t \in [0,T]} \| \theta(t) \|_{H^k}^2 \int_0^T \| \nabla u_2(t) \|_{L^\infty} \, dt \nonumber\\
&\quad + \int_0^T \| u(t) \|_{L^2}^2 \, dt \nonumber\\
&\lesssim \varepsilon^2 + M \varepsilon^2 \int_0^T \| \nabla u_2(t) \|_{L^\infty} \, dt + \mathcal{E}(0) \nonumber\\
&\lesssim \varepsilon^2 + M \varepsilon^2 \int_0^T \| \nabla u_2(t) \|_{L^\infty} \, dt,
\end{align}
where we have used estimates \eqref{energy_derivatives_1} and \eqref{energy_as_a_norm}.

To validate this bound for arbitrary $T$, it suffices to show:
\begin{align}\label{sia_2}
\int_0^T \| \nabla u_2(t) \|_{L^\infty} \,dt \ll 1,
\end{align}
with a bound independent of $T$. Suppose $T^*$ is the first time \eqref{assumption_main} fails. Then \eqref{sia_1} gives:
\[
M\varepsilon^2 = \sup_{t \in [0,T^*]} \| \theta(t) \|_{H^k}^2 + \int_0^{T^*} \| u(t) \|_{H^k}^2 \,dt 
\le C\varepsilon^2 + C\left( \int_0^T \| \nabla u_2(t) \|_{L^\infty} \,dt \right) M\varepsilon^2 < M\varepsilon^2,
\]
a contradiction, provided $M$ is chosen sufficiently large, independent of $T^*, \varepsilon$.

To verify \eqref{sia_2}, we use Sobolev embedding and interpolation for fixed $0 < \eta < k-2$:
\[
\| \nabla u_2(t) \|_{L^\infty} \le C_{\eta} \| u_2(t) \|_{H^{2+\eta}} \le C_{k,\eta} \| u_2(t) \|_{H^2}^{\frac{k-2-\eta}{k-2}} \| u(t) \|_{H^k}^{\frac{\eta}{k-2}}.
\]
From \eqref{usdd2es} and \eqref{uestimae1}, we \textcolor{black}{deduce that}
\[
\| u_2(t) \|_{H^2} \le \varepsilon O(t^{-k/2}), \quad \| u(t) \|_{H^k} \le (M\varepsilon^2)^{1/2} O(t^{-1/2}).
\]
Combining these, we obtain:
\[
\| \nabla u_2(t) \|_{L^\infty} \le o(M\varepsilon^2) O\left(t^{-\left(\frac{k(k-2-\eta)}{2(k-2)} + \frac{\eta}{2(k-2)}\right)}\right).
\]
Since $k - 2 > 0$, we may choose $\eta \in (0, k-2)$ such that
\[
\frac{k(k-2-\eta)}{2(k-2)} + \frac{\eta}{2(k-2)} > 1.
\]
For example, any $\eta$ satisfying $0 < \eta < \frac{(k-2)^2}{k-1}$ suffices. For such $\eta$, we conclude:
\[
\int_0^T \| \nabla u_2(t) \|_{L^\infty} \,dt = o(M\varepsilon^2),
\]
\textcolor{black}{thereby} verifying \eqref{sia_2} and completing the proof of Theorem~\ref{main_theorem}.
\color{black}

\subsection{Notations} We denote implicit constants arising in estimates by  $C$, which may vary from line to line. Such a constant may depend only on  $\gamma$, as defined in \eqref{alpha_def}. Otherwise, any dependence on other variables will be explicitly mentioned. For two quantities  $A$  and $ B$  satisfying  $A \leq C B$ , we write  $A \leq_C B$  for simplicity.

\subsection*{Data Availability \& Conflict of Interest Statements} Data sharing not applicable to this article as no datasets were generated or analysed during the current study.  The authors declare that they have no conflicts of interest.

\subsection*{Acknowledgements} RB and SW were partially supported  by the Italian Ministry of University and Research, PRIN 2022HSSYPN ``Turbulent Effects vs Stability in Equations from Oceanography'', PNRR Italia Domani, funded
by the European Union NextGenerationEU, CUP B53D23009300001. MJJ was partially supported by NSF DMS-2043024. JP was partially supported by SNSF Ambizione fellowship project PZ00P2-216083, the
Yonsei University Research Fund of 2024-22-0500, and  the POSCO Science Fellowship of POSCO
TJ Park Foundation. 
\section{Preliminaries}
 In this section, we collect preliminary tools of independent interest. 
\subsection{Fractional Sobolev spaces}
Here we recall the usual definition for the fractional Sobolev norms. For $f\in C^\infty(\Omega)$, we define
\[
\rVert f\rVert^2_{\dot{H}^k}:=\sum_{n=0}^\infty \int_{\mathbb{R}}\left(|n|^{2k}+|\xi|^{2k}\right)|\widehat{f}_n(\xi)|^2\,d\xi,
\]
where
\begin{align}\label{fourier_notation}
\widehat{f}_n(\xi):=\int_{\mathbb{R}}\int_{\mathbb{T}}f(x_1,x_2)e^{-i n x_1}e^{-ix_2\xi}dx_1dx_2.
\end{align}
We also denote
\[
\rVert f\rVert_{H^k}^2:=\rVert f\rVert_{L^2}^2+ \rVert f\rVert_{\dot{H}^k}^2.
\]
For $k>0$, we will use the following fractional Laplacians:
\begin{equation}\label{d_denf1}
\begin{aligned}
|D_1|^{k}f(x_1,x_2)&:=\sum_{n\in\mathbb{Z}}|n|^k f_n(x_2)e^{in x_1},\text{ where $f_n(x_2):=\int_{\mathbb{T}}f(x_1,x_2)e^{-i n x_1}dx_1$,}\\
|D_2|^{k}f(x_1,x_2)&:=\int_{\mathbb{R}}|\xi_1|^{k}\widehat{f}(x_1,\xi)e^{i\xi x_2}d\xi,\text{ where $\widehat{f}(x_1,\xi):=\int_{\mathbb{R}}f(x_1,x_2)e^{-i\xi x_2}dx_2$},\\
|D|^k f(x_1,x_2)&:=(|D_1|^k+|D_2|^k)f(x_1,x_2).
\end{aligned}
\end{equation}
\subsection{Time average decay}
 In this subsection, we recall useful inequalities arising in the subsequent energy estimates.
 
 \begin{lemma}\cite[Lemma 2.1]{park2024stability}\label{ABC_ODsdE}
Let $\alpha>0,\ 1<n$. Let  $a(t)$ and $f(t)$ be nonnegative functions on $[0,T]$ such that
\[
\frac{d}{dt}f(t) \le- a(t)^{-\alpha}f(t)^n,\quad f(0)=f_0.
\]
Then, $f$ satisfies
\[
f(t)\le_{\alpha,n} \frac{A^{\alpha/(n-1)}}{t^{(\alpha+1)/(n-1)}},\text{for all $t\in [0,T]$, where $A:=\int_0^t a(s)ds.$}
\]
\end{lemma}

\begin{lemma}\cite[Lemma 2.2]{park2024stability}\label{ODEsdsdBDC}
Let $n>0$ and $f(t),g(t)$ be nonnegative functions on $[0,T]$  such that
\begin{align}\label{ODEsure}
\frac{d}{dt}f(t)\le-g(t)\text{ and } f(t)\le \frac{C}{t^n},
\end{align}
for some $C>0$.
Then,  it holds that
\[
\frac{2}{t}\int_{t/2}^t g(s)ds\le_n \frac{C}{t^{n+1}}\text{ for all $t\in [0,T]$.}
\]
\end{lemma}

\begin{lemma}\label{Time_average_2}
Let $T> 0$ and $n> 1$. Let $f(t)$ be a nonnegative function on $[0,T]$ such that
\[
\frac{2}{t}\int_{t/2}^{t} f(s)ds\le \frac{A}{t^n},\text{ for some $A>0$, for all $t\in [0,T]$}.
\]
Then, for $0<t<T$, we have 
\[
\int_{t}^{T} f(t)ds \le C_{n}At^{1-n},
\]
where $C_{n}>0$ is a constant that depends  only on $n$.
\end{lemma}
\begin{proof}
  Given $T>t>0$, we fix $N\in\mathbb{N}$ so that
  \begin{align}\label{N_range}
  \frac{T}{2^{N+1}}\le t \le \frac{T}{2^N} \le 2t,
  \end{align}
  and we denote $T_i:= 2^{-i}T$ for $i=0,\ldots, N$. We estimate
  \begin{align*}
  \int_{t}^T f(t)dt &\le  \sum_{i=0}^N (T_i-T_{i+1})\frac{1}{T_i-T_{i+1}}\int_{T_{i+1}}^{T_i}f(t)dt\le \sum_{i=0}^{N}\frac{A(T_i-T_{i+1})}{T_i^n}\le A\sum_{i=0}^NT_i^{1-n}\\
  &\le AT^{1-n} \sum_{i=0}^N \left(2^{n-1}\right)^{i}= AT^{1-n}\frac{2^{(N+1)(n-1)}}{2^{n-1}-1}\le C_n A \left( \frac{T}{2^{N+1}}\right)^{1-n}\le C_n A \left( \frac{T}{2^{N}}\right)^{1-n}\\
  &\le C_n A t^{1-n},
  \end{align*}
  where the last inequality follows from \eqref{N_range}. This gives the desired result.
\end{proof}

\begin{lemma}\label{diff_lem_time}
Let $H,f,g,h$ be nonnegative functions satisfying for some $T>0$,
\[
\frac{d}{dt}H(t)\le -f(t) + h(t)H(t) +g(t), \text{ for $t\in [0,T]$}.
\]
Suppose there exist $A,n>0$ such that 
\[
\frac{2}{t}\int_{t/2}^t H(s) ds \le \frac{A}{t^{n}},\quad \frac{2}{t}\int_{t/2}^t g(s)ds \le \frac{A}{t^{n+1}},\text{  and }\int_0^T h(s)ds\le 1,
\]
for $t\in [0,T]$. Then there exists a constant $C>0$, which does not depend on $A$ nor $T$, such that
\[
\frac{2}{t}\int_{t/2}^t f(s) ds\le C\frac{A}{t^{n+1}} \text{ for $t\in [0,T]$}.
\]
\end{lemma}
\begin{proof}
Let $L(t):=e^{-\int_0^t h(s)ds}$.  Since $\int_0^T h(s)ds\le 1$, we have
\begin{align}\label{hdonotblowup}
 e^{-1}\le L(t)\le 1.
\end{align}
Furthermore, it holds that $\frac{d}{dt} \left(L(t)H(t)\right)\le -f(t)L(t) + g(t)L(t)$, yielding\[
L(t)H(t) - L(s)H(s) \le - \int_s^t f(u)L(u)du + \int_s^t g(u)L(u)du.
\]
Since $H,L\ge 0$, we have
\[
\int_s^t f(u)L(u)du \le L(s)H(s) + \int_s^t L(u)g(u)du.
\]
Integrating both sides over  $s\in [t/4,t]$, we get
\[
\int_{t/4}^t f(u)L(u) \left(u-\frac{t}4\right)du\le \int_{t/4}^t L(s)H(s)ds + \int_{t/4}^{t}L(u)g(u)\left( u-\frac{t}{4}\right)du
\]
From \eqref{hdonotblowup}, we obtain
\[
{\frac{t}{4e}\int_{t/2}^t f(s)ds \le \int_{t/4}^t f(u)L(u) \left(u-\frac{t}4\right)du} \le \int_{t/4}^t H(s)ds + t\int_{t/4}^t g(s)ds,
\]
implying that
\[
\frac{2}{t}\int_{t/2}^t f(s)ds \le \frac{C}{t^2}\int_{t/4}^t H(s)ds + \frac{C}{t}\int_{t/4}^t g(s)ds.
\]
Using the estimates for $H$ and $g$ in the assumption of the lemma, we obtain the desired result.
\end{proof}
\subsection{Stratification and potential energy}
In this subsection, we collect some basic properties of functions that are close to $\rho_s$ in the Sobolev space $H^k$, for $k > 2$. Since the domain $\Omega$ is two-dimensional, the Sobolev embedding $H^k(\Omega) \hookrightarrow C^1(\Omega)$ holds. Consequently, we expect that certain quantitative features of $\rho_s$ - being strictly decreasing - also hold for such nearby functions.

 Firstly, the inverse function theorem, combined with the assumption that $\inf_{x\in\Omega}\left(-\partial_2\rho_s(x_2)\right)>0$
 guarantees the existence of the inverse of $\rho_s$, denoted by  
 \[
 \phi_{\rho_s}(s):=\rho_s^{-1}(s) \text{ for $s\in\mathbb{R}$}.
 \] Moreover,  let us recall from the regularity assumption in \eqref{density_property}  that $\rVert \partial_2\rho_s\rVert_{C^{k+1}}<\infty$. Hence one can straightforwardly deduce, using the inverse function theorem, that 
\begin{align}\label{rkskd2sd}
\rVert \partial_s\phi_{\rho_s}\rVert_{C^3}+\rVert \partial_2\rho_s\rVert_{C^3}\le \rVert \partial_2\rho_s\rVert_{C^{k+1}}\le  C(\gamma),
   \end{align}
   where $\gamma$ is as in \eqref{alpha_def}.
   
   In the rest of this section, we will consider functions $f\in H^k(\Omega)$ that are close to $\rho_s$, namely
   \begin{align}\label{small_delta_assumption_2}
   \rVert f-\rho_s\rVert_{H^k}\le \delta \text{ for sufficiently small $\delta>0$}.
   \end{align}
    For a technical reason, to study the potential energy for such functions, we will also impose a decay assumption as well
   \begin{align}\label{decay_assumption_1}
    \lim_{|x_2|\to\infty} \sup_{x_1\in\mathbb{T}}|f(x_1,x_2)-\rho_s(x_2)|\sqrt{|x_2|}=0,
   \end{align}
   which will turn out to be true for sufficiently regular solutions to the IPM equation.

  \begin{lemma}\label{rearrangement_lem}
Let $f\in H^k(\Omega)$ and $\gamma$ be as in \eqref{alpha_def}. Then there exists $\delta=\delta(\gamma)>0$, such that if \eqref{small_delta_assumption_2} holds, then there exist $\phi_0:\mathbb{R}\mapsto \mathbb{R}$ and $h:\mathbb{T}\times \mathbb{R}\to \mathbb{R}$ satisfying
 \begin{align*}
\int_{\mathbb{T}}h(x_1,s)dx_1&=0\text{ and } f(x_1,\phi_0(s)+h(x_1,s))= s,\text{ for $(x_1,s)\in \mathbb{T}\times \mathbb{R}$.}
 \end{align*}
 Furthermore, the following estimates hold:
 \begin{align}\label{regularity_level}
 \rVert \phi_0-\phi_{\rho_s}\rVert_{L^\infty}+\rVert\partial_s(\phi_0-\phi_{\rho_s})\rVert_{L^\infty} + \rVert h\rVert_{W^{1,\infty}}+\rVert h\rVert_{L^2}+\rVert \partial_sh\rVert_{L^2}\le C(\gamma) \rVert f-{\rho}_s\rVert_{H^k}.
 \end{align} 
   In addition, if \eqref{decay_assumption_1} holds, then
  \begin{align}\label{decay_estimate_1}
  \lim_{|s|\to\infty} \sup_{x_1\in\mathbb{T}}\sqrt{|s|}|h(x_1,s)|=0.
  \end{align}
  \end{lemma}
   \begin{proof} 
    We first notice that $x_2\to f(x_1,x_2)$ is invertible for each $x_1\in\mathbb{T}$. Indeed, using $H^k(\Omega)\hookrightarrow C^1(\Omega)$ for $k>2$, we have 
   \begin{align}\label{monotone1}
  \partial_2 f (x_1,x_2)&= \partial_2{\rho}_s(x_2) + (\partial_2 f(x_1,x_2) - \partial_2 {\rho}_s(x_2))\le  -\gamma + C\rVert f-{\rho}_s\rVert_{H^k}\le  -\gamma+C\delta < -\gamma/2,
  \end{align}
  for sufficiently small $\delta>0$.
   Hence,  $x_2\mapsto f(x_1,x_2)$ is  strictly decreasing.  The implicit function theorem tells us that there is a parametrization of each level curve, $\left\{f=s\right\}$, which we will denote by $\phi(\cdot,s)$, that is,
  \begin{align}\label{level_set_ff}
  f(x_1,\phi(x_1,s))=s\text{ for each $(x_1,s)\in \mathbb{T}\times \mathbb{R}$.}
  \end{align} Let us decompose it as
  \begin{align}\label{phi_decomposition}
  \phi(x_1,s) = \phi_{\rho_s}(s) + g(x_1,s), \text{ where } g(x_1,s):=\phi(x_1,s) -\phi_{\rho_s}(s),
   \end{align}
   where $\phi_{\rho_s}(s):=\rho_s^{-1}(s)$.
 We claim the following estimates:
 \begin{align}
 \rVert g\rVert_{L^\infty}&\le C \rVert f-\rho_s\rVert_{H^k}, \label{gsmall}\\
 \rVert \nabla g\rVert_{L^\infty}&\le C \rVert f-\rho_s\rVert_{H^k}, \label{rjjsdwdqdsd2}\\
\rVert g\rVert_{L^2}&\le C\rVert f-\rho_s\rVert_{H^k}, \label{justr}\\
\rVert \partial_sg\rVert_{L^2}&\le C \rVert f-\rho_s\rVert_{H^k}. \label{amy}
 \end{align}
 For now, let us assume the validity of the above estimates. Towards the proof of the lemma,  we define $\phi_0$ and $h$ as follows:
  \[
  \phi_0(s) := \phi_{\rho_s}(s) + \frac{1}{2\pi}\int_{\mathbb{T}}g(y,s)dx_1,\quad h(x_1,s):=g(x_1,s)-\int_{\mathbb{T}}g(y,s)dy.
  \]
   By its definition, we have  $\int_{\mathbb{T}}h(x_1,s)dx_1=0$ for each $s\in \mathbb{R}$.  Also, \eqref{level_set_ff} and \eqref{phi_decomposition} tell us that
 \[
 f(x_1,\phi_0(s)+h(x_1,s)) = f(x_1,\phi_{\rho_s}(s)+g(x_1,s)) =f (x_1,\phi(x_1,s)) = s \text{ for all $s\in\mathbb{R}$}.
 \]
Collecting the estimates for $g$  in \eqref{gsmall}-\eqref{amy},  we see that \begin{align*}
\rVert h\rVert_{W^{1,\infty}},\ \rVert \partial_s h\rVert_{L^{2}},\ \rVert h\rVert_{L^2},\  \rVert \partial_s(\phi_0-\phi_{\rho_s})\rVert_{L^{\infty}},\ \rVert \phi_0-\phi_{\rho_s}\rVert_{L^\infty}\le C\rVert f-{\rho}_s\rVert_{H^k},
 \end{align*}
 which proves \eqref{regularity_level}.

 Now, let us derive the estimates \eqref{gsmall}-\eqref{amy}. 

\textbf{Estimate for $\rVert g\rVert_{L^\infty}$.}
Writing
      \begin{align*}
     0  = \rho_s(\phi_{\rho_s}(s))-f(x_1,\phi(x_1,s)) = \left(\rho_s(\phi_{\rho_s}(s)) - \rho_s(\phi(x_1,s))\right)+ \left(\rho_s(\phi(x_1,s)) -f(x_1,\phi(x_1,s))
     \right),
     \end{align*}
     we notice that
     \begin{align}\label{betterex}
    \int_{0}^{g(x_1,s)}\partial_2\rho_s(y+\phi_{\rho_s}(s))dy=  \rho_s(\phi(x_1,s)) -f(x_1,\phi(x_1,s)).
     \end{align}   
   Since $\partial_2\rho_s<-\gamma<0$ and $ \rVert f-\rho_s\rVert_{L^\infty}\le C \rVert f-\rho_s\rVert_{H^k}$, the identity \eqref{betterex} gives us \eqref{gsmall}. 
 
\textbf{Estimate for $\rVert\nabla  g\rVert_{L^\infty}$.}
Differentiating \eqref{betterex}, we obtain
\begin{align}
\partial_1g(x_1,s)\partial_2\rho_s(\phi(x_1,s))& = -\partial_1f(x_1,\phi(x_1,s)) + (\partial_2\rho_s-\partial_2f)(x_1,\phi(x_1,s))\partial_1\phi(x_1,s) \label{tlqkfeljsd12sd2}\\
\partial_s g(x_1,s)\partial_2\rho_s (\phi(x_1,s))& +  \int_{0}^{g(x_1,s)}\partial_{22}\rho_s(y+\phi_{\rho_s}(s))\partial_s\phi_{\rho_s}(s)dy \nonumber\\
& \quad = (\partial_2\rho_s - \partial_2f)(x_1,\phi(x_1,s))\partial_s\phi(x_1,s).\label{tlqkfeljsd12sd}
\end{align}
Again using $\partial_2\rho_s<-\gamma<0$, we have the following pointwise estimates:
\begin{align}\label{lower_bound_g}
|\partial_1 g\partial_2\rho_s |\ge C|\partial_1 g|,\quad |\partial_s g\partial_2\rho_s |\ge C | \partial_sg|.
\end{align} On the other hand,  we can estimate  the integral on the left-hand side of \eqref{tlqkfeljsd12sd} as
\begin{align}\label{cwnaj1}
\left| \int_{0}^{g(x_1,s)}\partial_{22}\rho_s(y+\phi_{\rho_s}(s))\partial_s\phi_{\rho_s}(s)dy\right|\le | g(x_1,s)|\rVert \partial_2\rho_s\rVert_{C^1}\rVert \phi_{\rho_s}\rVert_{C^1}\le  C\rVert f-\rho_s\rVert_{H^k},
\end{align}
where the last inequality follows from \eqref{rkskd2sd} and \eqref{gsmall}. For the right-hand side of \eqref{tlqkfeljsd12sd2}, we use $H^k\hookrightarrow C^1$ and derive
\begin{align}
\rVert \partial_1f\rVert_{L^\infty} = \rVert \partial_1f - \partial_1\rho_s\rVert_{L^\infty}&\le C \rVert f-\rho_s\rVert_{H^k}\label{stonecold_2},
\end{align}
and 
\begin{align}
\rVert(\partial_2\rho_s - \partial_2f)\partial_1\phi\rVert_{L^\infty}&\le C \rVert \rho_s-f\rVert_{H^k}\rVert\partial_1\phi\rVert_{L^\infty} \le C \rVert \rho_s-f\rVert_{H^k} \rVert\partial_1 g\rVert_{L^\infty} \le C  \delta \rVert\partial_1 g\rVert_{L^\infty}.\label{stone_cold3}
\end{align}
Similarly, we have the following pointwise estimates
\begin{align}
|(\partial_2\rho_s - \partial_2f)\partial_s\phi|&\le C | \partial_2(\rho_s-f)||\partial_s\phi|  \le C| \partial_2(\rho_s-f)| |\partial_s\phi_{\rho_s}| + C\delta |\partial_s g|\label{stone_cold242}\\
& \le C \rVert \rho_s-f\rVert_{H^k} + C\delta |\partial_s g|,\label{stone_cold2}
\end{align}
where the last inequality is due to \eqref{rkskd2sd}.
We plug the first estimate in \eqref{lower_bound_g}, \eqref{stonecold_2} and \eqref{stone_cold3} into \eqref{tlqkfeljsd12sd2}, yielding that
\[
C|\partial_s g(x_1,s)|\le C\rVert f-\rho_s\rVert_{H^k}+ C\delta \rVert \partial_1g\rVert_{L^\infty},\text{ for every $(x_1,s)\in\mathbb{T}\times\mathbb{R}$.}
\]
Hence, $\rVert \partial_1 g \rVert_{L^\infty}\le_C \rVert f-\rho_s\rVert_{H^k}+\delta \rVert \partial_1 g\rVert_{L^\infty}$. Assuming $\delta$ is sufficiently small, we obtain
\begin{align}\label{rjjsdwdqdsd}
\rVert \partial_1g\rVert_{L^\infty}\le C \rVert f-\rho_s\rVert_{H^k}.
\end{align}
Similarly, plugging the second estimate in \eqref{lower_bound_g}, \eqref{cwnaj1} and \eqref{stone_cold2} into \eqref{tlqkfeljsd12sd} gives us  
\[
\rVert \partial_sg\rVert_{L^\infty}\le C \rVert f-\rho_s\rVert_{H^k}.
\] Therefore we obtain \eqref{rjjsdwdqdsd2}.

\textbf{Estimate for $\rVert g\rVert_{L^2}$.} In \eqref{betterex}, it follows from $-\partial_2\rho_s >\gamma$ that 
\[
|g(x_1,s)|\le C|\rho_s(\phi(x_1,s)) -f(x_1,\phi(x_1,s))|.
\]
Using the change of variables $s\to f(x_1,x_2)$, we compute
\begin{align*}
\int_{\mathbb{T}\times\mathbb{R}}|g(x_1,s)|^2\,dx_1ds& \leq C \int_{\mathbb{T}\times\mathbb{R}}|\rho_s(\phi(x_1,s)) -f(x_1,\phi(x_1,s))|^2\, dx_1ds \\
&= C\int_{\mathbb{T}\times\mathbb{R}} |\rho_s(\phi(x_1,f(x_1,x_2)))-f(x_1,\phi(x_1,f(x_1,x_2)))|^2 |\partial_2f (x_1,x_2)|\,dx_1dx_2\\
&\le \int_{\mathbb{T}\times\mathbb{R}}|\rho_s(x)-f(x)|^2|\partial_2f(x)|dx\\
&\le C\rVert f-\rho_s\rVert_{L^2}^2\rVert \partial_2f\rVert_{L^\infty}\\
&\le  C \rVert f-\rho_s\rVert_{L^2}^2\left( \rVert \partial_2f-\partial_2\rho_s\rVert_{L^\infty} + \rVert \partial_2\rho_s\rVert_{L^\infty}\right)\\
&\le C \rVert f-\rho_s\rVert_{L^2}^2,
\end{align*}
where the last inequality follows from \eqref{rkskd2sd} and $\rVert \partial_2(f-\rho_s)\rVert_{L^\infty}\le_C\rVert f-\rho_s\rVert_{H^k}\le \delta\ll 1$.
Therefore we can obtain \eqref{justr}.

\textbf{Estimate for $\rVert \partial_sg\rVert_{L^2}$.}   The $L^2$ norm in $(x_1,s)$ of the left-hand side in \eqref{tlqkfeljsd12sd} can be estimated as
\begin{align*}
\Big\rVert \partial_s g(x_1,s)\partial_2\rho_s (\phi(x_1,s))& +  \int_{0}^{g(x_1,s)}\partial_{22}\rho_s(y+\phi_{\rho_s}(s))\partial_s\phi_{\rho_s}(s)dy \Big\rVert_{L^2}  \\
&\ge C\rVert \partial_sg\rVert_{L^2} - \rVert g\rVert_{L^2}\rVert \partial_{22}\rho_s\rVert_{L^\infty}\rVert \partial_s \phi_{\rho_s}\rVert_{L^\infty}\\
&\ge \rVert \partial_sg\rVert_{L^2}- \rVert f-\rho_s\rVert_{H^k},
\end{align*}
where the last inequality follows from \eqref{rkskd2sd} and \eqref{justr}.
For the right-hand side of \eqref{tlqkfeljsd12sd}, the pointwise estimate in \eqref{stone_cold242} gives us
\[
\rVert (\partial_2\rho_s - \partial_2f)\partial_s\phi\rVert_{L^2}\le C \rVert \rho_s-f\rVert_{H^k}\rVert \partial_s\phi_{\rho_s}\rVert_{L^\infty} +C \delta \rVert \partial_s g\rVert_{L^2}\le C \rVert \rho_s-f\rVert_{H^k}+C \delta \rVert \partial_s g\rVert_{L^2},
\]
where the last inequality follows from \eqref{rkskd2sd}. Hence, plugging these estimates into \eqref{tlqkfeljsd12sd}, we obtain \eqref{amy}.

Lastly, let us assume \eqref{decay_assumption_1} and aim to prove \eqref{decay_estimate_1}. 
From \eqref{betterex}, we write
  \begin{align}\label{amymacdonald}
 \sqrt{s} |g(x_1,s)|&\le C\sqrt{s} |\rho_s(\phi(x_1,s)) -f(x_1,\phi(x_1,s))|\nonumber\\
 & = C\frac{\sqrt{|s|}}{\sqrt{|\phi{(x_1,s)}|}}\left(\sqrt{|\phi(x_1,s)|}|\rho_s(\phi(x_1,s)) -f(x_1,\phi(x_1,s))|\right).
    \end{align}
    Since $C^{-1}<|\partial_s\phi_{\rho_s}(s)|< C$ for all $s\in\mathbb{R}$, the estimate \eqref{rjjsdwdqdsd2} tells us that for sufficiently small $\delta>0$, 
    \[
    C^{-1}<\left|\partial_s\phi(x_1,s)\right|< C \text{ for all $s\in\mathbb{R}$}.\]
    Therefore,
    \[
   \sup_{(x_1,s)\in\mathbb{T}\times\mathbb{R}} \sqrt{\frac{|s|}{|\phi(x_1,s)|}}< C.
    \]
    Since this also implies $\lim_{s\to\infty} \inf_{x_1\in\mathbb{T}}|\phi(x_1,s)|=\infty$, \eqref{decay_assumption_1} reads
    \[
    \lim_{|s|\to\infty}\left(\sqrt{|\phi(x_1,s)|}|\rho_s(\phi(x_1,s)) -f(x_1,\phi(x_1,s))|\right) = 0,
    \]
    thus \eqref{amymacdonald} reads
    \[
    \lim_{|s|\to\infty}\sup_{x_1\in\mathbb{T}}\sqrt{s} |g(x_1,s)| =0.
    \]
  Then \eqref{decay_estimate_1} follows from the definition of $h$.
  \end{proof} 
  
  Under the assumption in Lemma~\ref{rearrangement_lem}, the estimate for $\partial_s\phi_0-\partial_s\phi_{\rho_s}$ in \eqref{regularity_level} gives us 
  \[
 \partial_s\phi_0(s) \le |\partial_s\phi_0(x) - \partial_s\phi_{\rho_s}| +  \partial_s\phi_{\rho_s}\le \delta -\gamma < -\frac{\gamma}{2}<0,
  \]
  for sufficiently small $\delta>0$. Therefore, the inverse function theorem ensures that the measure-preserving stratification of $f$ is well-defined:
  \[
  f^*(x)=f^*(x_2):=\phi_0^{-1}(x_2).
  \]
  
  Now, we define the potential energy $\mathcal{E}$ as
  \begin{align}\label{potential_e}
  \mathcal{E}(f):=\lim_{s\to\infty}\left(E_s(f)-E_s(f^*)\right),\text{ where }E_s(f):=\int_{\left\{ x\in\Omega: -s<f(x)<s\right\}}f(x)x_2dx,
  \end{align}
  whenever the limit exists.
     \begin{proposition}\label{propoos}
Let $f\in H^k(\Omega)$ and $\gamma$ be as in \eqref{alpha_def}. Then there exists $\delta=\delta(\gamma)>0$,  such that if \eqref{small_delta_assumption_2} and \eqref{decay_assumption_1} hold, then 
  $\mathcal{E}(f)$ is well-defined and
\begin{align}\label{energy_estimate_prop}
C^{-1}\rVert f-f^*\rVert_{L^2}^2\le \mathcal{E}(f)\le C\rVert f-f^*\rVert_{L^2}^2,
\end{align}
where $f^*$ is the measure-preserving stratification of $f$. Furthermore, denoting
\[
v=\nabla^\perp(-\Delta)^{-1}\partial_1f,
\]
we have
\begin{align}\label{u_to_ffstar}
\mathcal{E}(f)\le C\left(\rVert v\rVert_{L^2}^{\frac{k-1}{k}}\rVert v\rVert_{H^k}^{\frac{1}{k}}\right)^2.
 \end{align}
  \end{proposition}
  \begin{proof}
  We claim 
  \begin{align}\label{hatedogs}
 \rVert \partial_1f \rVert_{L^2(\Omega)}\ge C\rVert f-f^*\rVert_{L^2(\Omega)},
 \end{align}
 \textbf{Proof of \eqref{hatedogs}.} For sufficiently small $\delta>0$,  Lemma~\ref{rearrangement_lem} ensures the existence of  $\phi_0,h$ such that 
  \begin{align}\
 \int_{\mathbb{T}}h(x_1,s)dx_1 &= 0,\quad    f(x_1,\phi_0(s)+h(x_1,s))=s,\label{level_32}
  \end{align}
  and
  \begin{align}\label{level_estimate32}
 \rVert \phi_0-\phi_{\rho_s}\rVert_{L^\infty}+ \rVert\partial_s\phi_0 -\partial_s\phi_{\rho_s}\rVert_{L^\infty} + \rVert h\rVert_{W^{1,\infty}} + \rVert h\rVert_{L^2}+ \rVert \partial_sh\rVert_{L^2}\le C\delta,
  \end{align}
  where $\phi_{\rho_s}$ is the inverse of $\rho_s$. From \eqref{level_estimate32}, it follows that
  \begin{align}
  -\partial_s\phi_0(x_2) = \left(-\partial_s\phi_0(x_2) + \partial_s \phi_{\rho_s}\right) - \partial_s\phi_{\rho_s}\ge \gamma -C\delta > \frac{\gamma}{2},\label{lowerbs}\\
   C^{-1}\le |\partial_s \phi(x_1,s)|\le C,\quad \text{ for $\phi(x_1,s):=\phi_0(x_1,s)+h(x_1,s)$,}\label{pon2s}\\
    |\partial_s\phi_0(s)|\le |\partial_s\phi_0(s) - \partial_s \phi_{\rho_s}(s)| + |\partial_s\phi_{\rho_s}(s)|\le_C \delta + \rVert \partial_s\phi_{\rho_s}\rVert_{C^2}< C, \label{pho_uppper}
  \end{align}
   where the last inequality is due to \eqref{rkskd2sd}. Therefore, the inverse function theorem guarantees that there exists $f^*$ such that
  \[
  f^*(x_2):=\left(\phi_0\right)^{-1}(x_2).
  \]
   Moreover, since for any  $s\in\mathbb{R}$, $ 0=f^*(\phi_0(s))-\rho_s(\phi_{\rho_s}(s))$,  we obtain
  \begin{align*}
0&=\partial_2f^*(\phi_0(s))\partial_s\phi_0(s) - \partial_2\rho_s(\phi_{\rho_s}(s))\partial_s\phi_{\rho_s}(s)\\
& = \left(\partial_2f^*(\phi_0(s))-\partial_2\rho_s(\phi_0(s))\right)\partial_s\phi_0(s)+\left(\partial_2\rho_s(\phi_0(s))-\partial_2\rho_s(\phi_{\rho_s}(s))\right)\partial_s\phi_0(s)\\
& \ + \partial_2\rho_s(\phi_{\rho_s}(s))\left( \partial_s\phi_0(s) - \partial_s\phi_{\rho_s}(s)\right).
  \end{align*}
 Using \eqref{alpha_def} and \eqref{level_estimate32}, this computation gives us
 \begin{align*}
\left|\left(\partial_2f^*(\phi_0(s))-\partial_2\rho_s(\phi_0(s))\right)\partial_s\phi_0(s)\right|&\le_C \rVert \partial_2\rho_s\rVert_{C^1}\rVert \phi_0 - \phi_{\rho_s}\rVert_{L^\infty}\rVert \partial_s\phi_0\rVert_{L^\infty}\\
&\ +\rVert \partial_2\rho_s\rVert_{L^\infty}\rVert \partial_s\phi_0 - \partial_s\phi_{\rho_s}\rVert_{L^\infty}\\
&\le C\delta,
 \end{align*}
 where the last inequality follows from \eqref{rkskd2sd}, \eqref{level_estimate32} and \eqref{pon2s}. Thus, using \eqref{lowerbs}, we get
 \begin{align}\label{fs2t2ar1}
 \rVert \partial_2 f^* - \partial_2 \rho_s\rVert_{L^\infty}\le C\delta.
 \end{align}
 Hence, for sufficiently small $\delta>0$, 
 \begin{align}\label{f_estimate1justtwo}
\inf_{x\in\Omega}-\partial_2f^*(x)>C>0,\quad \rVert \partial_2f^*\rVert_{L^\infty}<C.
 \end{align} 
 Using  the change of variables, $x_2\to \phi(x_1,s)$, we have
  \begin{align}
   \rVert f-f^*\rVert_{L^2(\Omega)}^2 &= \int_{\mathbb{T}}\int_\mathbb{R}|f(x_1,\phi(x_1,s))-f^*(\phi(x_1,s))|^2|\partial_s\phi(x_1,s)|dx_1ds.\label{L2est}
  \end{align}
 Thanks to \eqref{level_32}, we have
 \begin{align}\label{fatagain}
 f(x_1,\phi(x_1,s))&-f^*(\phi(x_1,s))\nonumber\\
 & =s- f^{*}(\phi_0(s)) + \left(f^*(\phi_0(s)) -f^*(\phi_0(s) + h(x_1,s))\right)\nonumber\\
 & = f^*(\phi_0(s)) -f^*(\phi_0(s) + h(x_1,s))\nonumber\\
 & = -\int_0^{h(x_1,s)}\partial_2f^*(\phi_0(s)+z)dz.\nonumber
 \end{align}
 Thus, using \eqref{f_estimate1justtwo} and \eqref{pon2s}, we get
 \[
 |h(x_1,s)| \le_C |f(x_1,\phi(x_1,s))-f^*(\phi(x_1,s))| \sqrt{|\partial_s\phi(x_1,s)|}\le_C |h(x_1,s)|.
 \]
  Plugging this into  \eqref{L2est}, we obtain \begin{align}\label{hpt}
C^{-1}\rVert h\rVert_{L^2(\mathbb{T}\times [0,1])}\le \rVert f-f^*\rVert_{L^2(\Omega)}\le C\rVert h\rVert_{L^2(\mathbb{T}\times[0,1])}.
 \end{align}

 On the other hand, differentiating \eqref{level_32} in $x_1$ gives us
  \[
  0 = \partial_1 (f(x_1,\phi(x_1,s))) = \partial_1f(x_1,\phi(x,s))+\partial_2f(x_1,\phi(x_1,s))\partial_1h(x_1,s).
  \]
 Similarly, diffierentiating the second identity in \eqref{level_32} yields  $ 1=  \partial_2 f(x_1,\phi(x_1,s))\partial_s\phi(s)$, thus 
  \[
\partial_1 f(x_1,\phi(x_1,s)) =-\frac{\partial_1h(x_1,s)}{\partial_s\phi(x_1,s)}.
  \]
 Then, using the change of variables $x_2\to \phi(x_1,s)$, we obtain
 \begin{align*}
 \rVert \partial_1 f\rVert_{L^2(\Omega)}^2&= \int_{\mathbb{T}\times\mathbb{R}}|\partial_1f(x_1,\phi(x_1,s))|^2|\partial_s\phi(x_1,s)|dx_1ds\\
 & \ge \int_{\mathbb{T}\times\mathbb{R}}\frac{|\partial_1 h(x_1,s)|^2}{|\partial_s \phi(x_1,s)|}dx_1ds\\
 &\ge  C\rVert \partial_1h \rVert_{L^2}^2\\
 &\ge C\rVert h\rVert_{L^2}^2
 \end{align*}
 where the second inequality is due to \eqref{pon2s} and the last inequality follows from the zero-average in $x_1$ of $h$ in \eqref{level_32} and the Poincar\'e inequality. Therefore, combining this with \eqref{hpt}, we conclude that \eqref{hatedogs} holds.
 
 \textbf{Proof of \eqref{energy_estimate_prop}.}Now, let us assume \eqref{decay_assumption_1}. Note that $E_s(f)$  can be written as
 \begin{align*}
 E_s(f)& = \int_{\mathbb{T}}\int_{\phi(x_1,s)}^{\phi(x_1,-s)} f(x_1,x_2)x_2 dx_2dx_1 \\
 &= -\int_{\mathbb{T}}\int_{-s}^{s} f(x_1,\phi(x_1,t))\phi(x_1,t)\partial_{t}\phi(x_1,t)dtdx_1\\
 &= -\int_{\mathbb{T}}\int_{-s}^s t \frac{1}{2}\partial_t\left(\phi(x_1,t)^2 \right)dtdx_1\\
 & = -\frac{1}{2}\int_{\mathbb{T}} s\left(\phi(x_1,s)^2 +\phi(x_1,-s)^2\right)dx_1 + \frac{1}{2}\int_{\mathbb{T}}\int_{-s}^s|\phi(x_1,t)|^2dtdx_1.
 \end{align*} The same computations gives us
 \[
 E_s(f^*)= -\frac{1}{2}\int_{\mathbb{T}} s\left(\phi_0(s)^2 +\phi_{0}(-s)^2\right)dx_1 + \frac{1}{2}\int_{\mathbb{T}}\int_{-s}^s|\phi_{0}(t)|^2dtdx_1.
 \]
 Hence, we get
 \begin{align}\label{energy_limit_s}
 E_s(f)-E_s(f^*) = -\frac{1}{2}\int_{\mathbb{T}}s\left(|h(x_1,s)|^2 +|h(x_1,-s)|^2\right) dx_1 +\frac{1}{2}\int_{\mathbb{T}}\int_{-s}^s |h(x_1,t)|^2dtdx_1,
 \end{align}
 where we used $\int_{\mathbb{T}}h(x_1,s)dx_1=0$ for each $s$.
 Since $h\in L^2(\mathbb{T}\times\mathbb{R})$, we have
 \[
 \lim_{s\to\infty}\int_{\mathbb{T}}\int_{-s}^s |h(x_1,t)|^2dtdx_1 = \rVert h\rVert_{L^2}^2,
 \]
 while \eqref{decay_estimate_1} of Lemma~\ref{rearrangement_lem} yields
 \[
 \lim_{s\to\infty} \int_{\mathbb{T}}s\left(|h(x_1,s)|^2 +|h(x_1,-s)|^2\right) dx_1  = 0.
 \]
 Therefore taking $s\to\infty$ in \eqref{energy_limit_s}, we obtain
 \[
 \mathcal{E}(f) = \frac{1}{2}\rVert h\rVert_{L^2(\mathbb{T}\times\mathbb{R})}^2.
 \]
 Combining this with \eqref{hpt}, we obtain \eqref{energy_estimate_prop}. 
 
 \textbf{Proof of \eqref{u_to_ffstar}.} Lastly, we use the usual Sobolev interpolation theorem to obtain
 \[
 \rVert \partial_1 f\rVert_{L^2}\le\rVert v\rVert_{H^{1}}\le C\rVert v\rVert_{L^2}^{\frac{k-1}{k}}\rVert v\rVert_{H^k}^{\frac{1}{k}}.
 \]
On the other hand, it follows from \eqref{hatedogs} and \eqref{energy_estimate_prop} that
\[
\mathcal{E}(f)\le \rVert f-f^*\rVert_{L^2}^2\le C\rVert \partial_1 f\rVert_{L^2}^2.
\]
 Combining these two inequalities, we conclude \eqref{u_to_ffstar}.
  \end{proof}
  
   We prove stability of measure-preserving stratification.
   \begin{lemma}\label{Stability_1}
   Let $f_n, f $ be  functions such that $\lim_{n\to\infty}\rVert f_n -  f\rVert_{H^k}=0$. If \eqref{small_delta_assumption_2} holds for $f_n,f$, then we have
   \[
   \lim_{n\to \infty}\rVert f^*-f_n^*\rVert_{L^2}=0.
   \]
   \end{lemma}
   \begin{proof}
   Applying \eqref{rearrangement_lem}, we can find $\phi_0,h_0$ and $\phi_{0,n},\ h_{0,n}$ such that
   \[
   s = f_n(x_1,\phi_{0,n}(s)+h_{n}(x_1,s)) =  f(x_1,\phi_{0}(s)+h(x_1,s)),\text{ for $(x_1,s)\in\mathbb{T}\times\mathbb{R}$,}
   \]
   satisfying
   \begin{align}\label{necessary_estiamted}
   0=\int_{\mathbb{T}}h(x_1,s)dx_1=\int_{\mathbb{T}}h_n(x_1,s)dx_1. 
   \end{align}
   As before, the assumption that \eqref{small_delta_assumption_2} holds for $f_n,f$ implies
   \begin{align}\label{f_esnonece}
   \rVert \partial_2f_n\rVert_{L^\infty},\rVert \partial_2f\rVert_{L^\infty}\le C.
   \end{align}
Furthermore, the estimates \eqref{lowerbs} reads 
\begin{align}\label{please_w}
|\partial_s\phi_0|,\ |\partial_s\phi_{0,s}|>C,\text{ therefore, }\rVert \partial_2f^*\rVert_{L^\infty},\rVert \partial_2f^*_n\rVert_{L^\infty}<C.
\end{align}   

Denoting $\phi:=\phi_0+h$ and $\phi_n:=\phi_{0,n}+h_n$, we have
   \begin{align}\label{carpenters}
   f(x_1,\phi(x_1,s))-f(x_1,\phi_{n}(x_1,s)) = f(x_1,\phi_n(x_1,s))-f_n(x_1,\phi_n(x_1,s)).
   \end{align}
   The left-hand side can be written as
   \[
    f(x_1,\phi(x_1,s))-f(x_1,\phi_{n}(x_1,s)) = \int_{\phi_n(x_1,s)}^{\phi(x_1,s)}\partial_2f(x_1,y)dy.
   \]
   Again using that $\partial_2f$ does not vanish, we have
   \begin{align}\label{nxxece_1sd}
   |f(x_1,\phi(x_1,s))-f(x_1,\phi_n(x_1,s))|\ge C|\phi(x_1,s)-\phi_n(x_1,s)|,\text{ for $(x_1,s)\in\mathbb{T}\times\mathbb{R}$.}
   \end{align}
   On the other hand, the change of variables $s\to f_n(x_1,x_2)$ for each fixed $x_1$ gives us
   \begin{align*}
   \int_{\mathbb{T}\times\mathbb{R}}| f(x_1,\phi_n(x_1,s))-f_n(x_1,\phi_n(x_1,s))|^2 dx_1ds&\le \int_{\mathbb{T}\times\mathbb{R}}|f(x)-f_n(x)|^2|\partial_2f_n(x)|dx\\
   &\le C\rVert f-f_n\rVert_{L^2}^2,
   \end{align*}
   where the last inequality follows from \eqref{f_esnonece}.
   Therefore, \eqref{carpenters} reads $\rVert \phi-\phi_n\rVert_{L^2}\le C\rVert f-f_n\rVert_{L^2}$, in other words,
  \begin{align}\label{voice_1}
 &  \int_{\mathbb{T}\times\mathbb{R}}|(\phi_0(s)-\phi_{0,n}(s))+(h(x_1,s)-h_n(x_1,s))|^2 dx_1ds\le C\rVert f-f_n\rVert_{L^2}^2\nonumber\\
  &\implies  \int_{\mathbb{T}\times \mathbb{R}}|\phi_0(s)-\phi_{0,n}(s)|^2dx_1ds + \rVert h-h_n\rVert_{L^2}^2\le C\rVert f-f_n\rVert_{L^2}^2\nonumber\\
  & \implies  \int_{\mathbb{T}\times \mathbb{R}}|\phi_0(s)-\phi_{0,n}(s)|^2dx_1ds \le C\rVert f-f_n\rVert_{L^2}^2,
  \end{align}
  where the first implication is due to \eqref{necessary_estiamted}. 
  
  Now, we compute, using the change of variables $s\to f^*(y)$ and recalling that $f^*=\phi_0^{-1},\ f_n^*=\phi_{0,n}^{-1}$,
  \begin{align*}
  \int_{\mathbb{R}}|\phi_0(s) - \phi_{0,n}(s)|^2 ds&= \int_{\mathbb{R}}|y- \phi_{0,n}(f^*(y)) -\phi_{0,n}(f_n^*(y))+\phi_{0,n}(f_n^*(y))|^2 |\partial_2f^*(y)|dy\\
  & = \int_{\mathbb{R}}|\phi_{0,n}(f^*_n(y))-\phi_{0,n}(f^*(y))|^2 |\partial_2f^*(y)|dy\\
  &\ge_C \inf_{s}|\partial_s\phi_{0,n}(s)|\sup_{y\in\mathbb{R}}|\partial_2f^*(y)|\rVert f_n^*-f^*\rVert_{L^2}^2\\
  &\ge C \rVert f_n^*-f^*\rVert_{L^2}^2,
  \end{align*}
  where the last inequality follows from \eqref{please_w}. Therefore, combining this with \eqref{voice_1}, we obtain the desired result.
   \end{proof}
    
    \section{Local wellposedness for the IPM equation}

 Let $\rho(t)$ be a smooth solution to the IPM equation and let $\rho_s$ be a stratified density. Denoting
  \begin{align}\label{deviation}
  \theta(t):=\rho(t)-\rho_s,
  \end{align}
  one can easily find that the IPM equation \eqref{IPM}-\eqref{Darcy_law} reads
\begin{align}\label{IPM_theta}
 \begin{cases}
 \theta_t +u\cdot\nabla{\theta} = -\partial_2\rho_s u_2,\quad \nabla\cdot u=0,\quad u=-\nabla p -\colvec{0 \\ \theta},&\text{ in $\Omega$,}\\
 \lim_{|x|\to \infty}u(x)=0,\\
 \theta(0,x):=\theta_0(x):=\rho_0(x)-\rho_s(x_2),
 \end{cases}
 \end{align}
The stream function formulation~\eqref{stream_IPM} can be rewritten as
\begin{align}\label{stream_IPM_2}
\begin{cases}
\Delta \Psi = -\partial_1\theta & \text{ in $\Omega$,}\\
\lim_{|x|\to \infty}\Psi(x)= 0.
\end{cases}
\end{align}
Threfore, for $\theta\in H^k$, the velocity can be written as
\begin{align}\label{velocity_11}
\colvec{u_1 \\ u_2} =\nabla^\perp \Psi= \colvec{- (-\Delta)^{-1}\partial_{12}\theta  \\ (-\Delta)^{-1}\partial_{11}\theta}.
\end{align}

 \begin{proposition}\label{lwp_IPM}
 Let $k>2$ and $\theta_0\in H^k(\Omega)$. Then there exists $T=T(\rVert \theta_0\rVert_{H^k})>0$ such that the equation \eqref{IPM_theta} admits a unique solution $\theta\in C([0,T);H^k)$.   Furthermore, if we assume in addition that $\theta_0\in C^\infty_c(\Omega)$, then 
 \begin{align}\label{decay_estimate_123}
\rVert \theta x_2\rVert_{L^\infty} + \rVert u_2x_2\rVert_{L^\infty}< \infty ,\text{ for every $t\in [0,T)$.}
 \end{align}
 \end{proposition}
 \begin{proof}
  The local well-posedness in $H^k$, $k>2$ for the IPM equation has been established in \cite[Section 3]{MR2337005}, using the fact that $\rho(t)\in H^k$ gives a Lipschitz continuous velocity field in two-dimensional domain. Even though we are dealing with $\rho$ such that $\rho-\rho_s\in H^k$, the same proof can be easily adapted since the vector field depends only on $\partial_1\rho$ and $\rho_s$ is independent of $x_1$ variable. Since a solution $\rho$ such that $\rho-\rho_s\in H^k$ is a classical solution, the local-wellposedness for the IPM equation ensures that \eqref{IPM_theta} is also locally well-posed.
  
   In the rest of the proof, we focus on deriving the decay rate \eqref{decay_estimate_123}.
   Let $\theta_0\in C^\infty_c(\Omega)$ and let $\theta\in C([0,T); H^k)$ be a unique solution.  Let $L>1$ be a constant and $\varphi:\Omega\to\mathbb{R}$ be a smooth nonnegative monotone increasing function such that
  \[
  \varphi(x)=  \begin{cases}
  1& \text{ for $|x_2|\le 1$}\\
  x_2 & \text{ for $2\le |x_2|\le L$}\\
  L+1 &\text{ for $|x_2|\ge L+2$}
  \end{cases}
  \]
  with the following estimates:
  \begin{align}\label{test_function_estimate1}
  \rVert \partial_2\varphi\rVert_{C^2}\le 2.
  \end{align}
  Denoting $
  g(t,x):=\theta(t,x)\varphi(x), $ 
  we find that $g$ solves
  \[
  g_t + u\cdot\nabla g = u_2\frac{\partial_2\varphi}{\varphi}g -\partial_2\rho_s u_2 \varphi
  \]

Let us derive crude energy estimates. In what follows, the implicit constant $C$ will not depend on $L$. The $L^2$-norm can be estimated as
\[
\frac{d}{dt}\rVert g\rVert_{L^2}^2 \le \rVert u_2\rVert_{L^\infty}\rVert \frac{\partial_2\varphi}{\varphi}\rVert_{L^\infty}\rVert g\rVert_{L^2}^2 + \rVert \partial_2\rho_s\rVert_{L^\infty}\rVert u_2\varphi\rVert_{L^2}\rVert g\rVert_{L^2}\le C\left(\rVert u_2\rVert_{L^\infty}\rVert g\rVert_{L^2}^2 +\rVert g\rVert_{L^2}\rVert u_2\varphi\rVert_{L^2}\right).
\] 
  Using \eqref{velocity_11}, we notice that $u_2\varphi$ solves
  \[
  \Delta (u_2\varphi) =- \partial_{11}(\theta \varphi) + \partial_2 u_2 \partial_2\varphi +u_2\partial_{22}\varphi
  \]
  Hence, using \eqref{test_function_estimate1} and that $\int_{\mathbb{T}}u_2(x)\varphi(x)dx_1=0$, we get
 \begin{align}\label{u_2decay_1}
  \rVert u_2\varphi\rVert_{L^2}\le \rVert \Delta(u_2\varphi)\rVert_{L^2}\le C\left( \rVert g\rVert_{H^2} + \rVert u_2\rVert_{H^2}\right).
  \end{align}
  Therefore, we obtain
  \begin{align}\label{imagine_crude_1}
  \frac{d}{dt}\rVert g\rVert_{L^2}^2 \le C\left(\rVert u\rVert_{H^2}+1\right)\rVert g\rVert_{H^2}^2 + C\rVert u\rVert_{H^2}^2  \end{align}
 Now we estimate $\rVert \Delta g\rVert_{L^2}$ as
 \begin{align*}
 \frac{d}{dt}\rVert \Delta g\rVert_{L^2}^2 &= -\int \left(\Delta (u\cdot\nabla g) - u\nabla \Delta g\right) \Delta g dx + \int \Delta\left(u_2\frac{\partial_2\varphi}{\varphi}g -\partial_2\rho_s u_2 \varphi \right)\Delta g dx\\
 &\le \rVert u\rVert_{H^2}\rVert g\rVert_{H^2}^2 +  \int \Delta\left(u_2\frac{\partial_2\varphi}{\varphi}g -\partial_2\rho_s u_2 \varphi \right)\Delta g dx.
 \end{align*}
 where the last inequality follows from the usual commutator estimates. Using \eqref{test_function_estimate1} and the fact that $H^2(\Omega)$ is a Banach algebra, the integral in the above estimate can be estimated as
 \begin{align*}
 \left| \int \Delta\left(u_2\frac{\partial_2\varphi}{\varphi}g -\partial_2\rho_s u_2 \varphi \right)\Delta g dx\right|&\le \rVert u_2\rVert_{H^2}\rVert g\rVert_{H^2}^2 + \rVert \partial_2\rho_s\rVert_{H^2}\rVert u_2\varphi\rVert_{H^2}\rVert g\rVert_{H^2}\\
 &\le C\left( \rVert u_2\rVert_{H^2}\rVert g\rVert_{H^2}^2  +  \rVert u_2\rVert_{H^2}\rVert g\rVert_{H^2} + \rVert g\rVert_{H^2}^2 \right).
 \end{align*}
  Therefore, we get
  \[
  \frac{d}{dt}\rVert \Delta g\rVert_{L^2}^2 \le C\left( \rVert u_2\rVert_{H^2}\rVert g\rVert_{H^2}^2  +  \rVert u_2\rVert_{H^2}\rVert g\rVert_{H^2} + \rVert g\rVert_{H^2}^2 \right)\le C\left(\rVert u\rVert_{H^2}+1\right)\rVert g\rVert_{H^2}^2 + C\rVert u\rVert_{H^2}^2.
  \]
 Combining this with \eqref{imagine_crude_1} and denoting $G(t):=\left(\rVert g\rVert_{L^2}^2+\rVert \Delta g\rVert_{L^2}^2 \right)$, we arrive at
 \[
 \frac{d}{dt}G(t)\le C\left(\rVert u\rVert_{H^2}+1\right)G(t) + C\rVert u\rVert_{H^2}^2.
  \]
 Since $\rVert u(t)\rVert_{H^2}<\infty$ for $t\in [0,T)$, Gr\"onwall's inequality gives us that $G(t)<\infty$ for each $t\in [0,T)$, hence
 \[
 \rVert \theta(t) \varphi \rVert_{H^2}<\infty, \text{ for each $t\in [0,T)$.}
 \]  
 Since the upper bound is independent of $L$, we arrive at
 \[
 \rVert \theta(t)x_2\rVert_{H^2}<\infty \text{ for each $t\in [0,T)$.}
 \]
 Similarly, it follows from \eqref{u_2decay_1} that $\rVert u_2\varphi\rVert_{H^2}$ is bounded uniformly in $L$, yielding that
  \[
  \rVert u_2x_2\rVert_{H^2}<\infty.
  \]
  Hence, using the embedding $H^2\hookrightarrow L^\infty$, we can obtain the desired decay estimates for $\theta$ and $u_2$ in \eqref{decay_estimate_123}.  \end{proof}

      \section{High regularity estimates}
      
   In this section, we derive Sobolev norm estimates for $\theta$. We will frequently use the usual commutator estimate (e.g. \cite{MR951744})
   \begin{align}\label{standard_tamepp}
      \rVert |D_i|^k(fg) - f|D_i|^k g\rVert_{L^2}\le \rVert f\rVert_{H^k}\rVert g\rVert_{L^\infty} + \rVert \nabla f\rVert_{L^\infty}\rVert g\rVert_{H^{k-1}},\text{ for $i=1,2$.}
   \end{align}
   The estimates for the Sobolev norms of  $\theta$  and  $u$  will be obtained in a standard manner using Sobolev embedding theorems. However, due to the anisotropic structure, certain technical computations are required, which will be carried out using the Fourier transform. These estimates are presented separately in the next lemma.  
   \begin{lemma}\label{Nonlinear_estimate_lem1}
   Let $\theta\in H^k$ for some $k>0$, and let $u$ be as in \eqref{velocity_11}. Then the following estimates hold:
   \begin{align}
   \left| \int\left(|D_2|^k(u_2\partial_2 \theta) - u_2 |D_2|^k\partial_2\theta\right) |D_2|^k\theta dx\right|&\le_C \rVert \nabla u_2\rVert_{L^\infty}\rVert \theta\rVert_{H^k}^2 + \rVert u\rVert_{H^k}^2\rVert \theta\rVert_{H^k}, \label{non_lin_es_1}\\
   \left|\int_{\Omega}\partial_1u_2\partial_2\theta\partial_{111}\theta dx \right|&\le_C \Big\rVert \partial_1|D|^{-1}u_2\Big\rVert_{H^2}^2\rVert\theta\rVert_{H^k},\label{non_lin_es_2}
   \end{align}
   where $C$ depends only on $k>2$.
   \end{lemma}
   \begin{proof}
   Let us fix our notations. For $f,g\in L^2(\Omega)$ (we denote the Fourier coefficients of $f$ by $\widehat{f}_{n}(\xi)$, see \eqref{fourier_notation}),
 \begin{align}\label{fourier_convolution_no} 
 \rVert \widehat{f}\rVert_{L^p}^p:=\sum_{n\in\mathbb{Z}}\int_{\mathbb{R}}|\widehat{f}_n(\xi)|^pd\xi \text{ for $p>1$},\quad \left(\widehat{f}*\widehat{g}\right)_{n}(\xi):= \sum_{n_1\in\mathbb{Z}}\int_{\mathbb{R}} |\widehat{f}_{n_1}(\xi_1)| |\widehat{g}_{n-n_1}(\xi-\xi_1)|d\xi_1.
 \end{align} Note that we take absolute values in the integral inside  the definition of the convolution.

  Young's inequality for convolution tells us
 \begin{align}\label{young_Fourier_1}
 \rVert \widehat{f}*\widehat{g}\rVert_{L^r}\le \rVert \widehat{f}\rVert_{L^p}\rVert \widehat{g}\rVert_{L^q},\text{ for $1+\frac{1}{r}=\frac{1}{p}+\frac{1}{q}$ and  $r,p,q\in [1,\infty]$.}
 \end{align}
 Furthermore, for $p\in [1,2)$ and $\alpha>0$ such that $\frac{2\alpha p}{2-p}>2$, the H\"older inequality gives us
 \begin{align}
 \rVert \widehat{f}\rVert_{L^p}^p& =\sum_{n\in\mathbb{Z}}\int_{\mathbb{R}}\left|(1+|n|+|\xi|)^{\alpha}\widehat{f}\right|^p (1+|n|+|\xi|)^{-\alpha p}d\xi\nonumber\\
 &\le \left(\sum_{n\in\mathbb{Z}}\int_{\mathbb{R}}\left|(1+|n|+|\xi|)^{\alpha}\widehat{f}\right|^2 d\xi\right)^{\frac{p}2}\left( \sum_{n\in\mathbb{Z}}\int_{\mathbb{R}} (1+|n|+|\xi|)^{-\frac{2\alpha p}{2-p}}d\xi\right)^{\frac{2-p}{2}}\nonumber\\
 &\le C_{p,\alpha}\rVert f\rVert_{H^\alpha}^p. \label{Hoilder_fouydirsd}
 \end{align}
 Especially, when $p=1$ and $\alpha=k-1$, we have $\frac{2\alpha p}{2-p} = 2(k-1)>2$, hence
 \begin{align}\label{special_calse_1}
 \rVert \widehat{f}\rVert_{L^1}\le C\rVert f\rVert_{H^{k-1}}.
 \end{align}
 Now we start proving the estimates in the lemma.

 \textbf{Proof of \eqref{non_lin_es_1}.}
 We write
 \begin{align}\label{I2split}
 I&:=\int\left(|D_2|^k(u_2\partial_2 \theta) - u_2 |D_2|^k\partial_2\theta\right) |D_2|^k\theta dx\nonumber\\
  &= \int \left(|D_2|^k (u_2\partial_2\theta) - k \left(|D_2| u_2 \right)|D_2|^{k-1}\partial_2\theta - u_2|D_2|^{k}\partial_2\theta\right)|D_2|^k\theta dx \nonumber\\
 & \ + k\int |D_2|u_2 \left(|D_2|^{k-1}\partial_2\theta \right)|D_2|^k\theta dx\nonumber\\
 &=: I_{1}+I_{2}.
 \end{align}
It is trivial that $I_{2}$ can be estimated as
\begin{align}\label{I22estimate}
|I_{2}|\le \rVert \nabla u_2\rVert_{L^\infty}\rVert \theta\rVert_{H^k}^2.
\end{align}
For $I_{1}$, the integral can be expressed, in terms of the integrands' Fourier coefficients, as
\begin{align}\label{I21expression}
I_{1} = \sum_{n_1,n_2\in\mathbb{Z}}\int_{(\xi_1,\xi_2)\in\mathbb{R}^2} m(\xi_1,\xi_2)\widehat{u_2}_{n_1}(\xi_1)\widehat{\partial_2\theta}_{n_2}(\xi_2)\widehat{\theta}_{n_1+n_2}(\xi_1+\xi_2)d\xi_1d\xi_2,
\end{align}
where
\[
m(\xi_1,\xi_2):=\left(|\xi_1+\xi_2|^k -k|\xi_1||\xi_2|^{k-1}-|\xi_2|^k\right)|\xi_1+\xi_2|^k.
\]
Let us rewrite the expression in \eqref{I21expression}, using that $u_2=\partial_1\Psi$, as
\begin{align}\label{I21_stevie_wonder}
I_{1} = \sum_{n_1,n_2\in\mathbb{Z}}\int_{(\xi_1,\xi_2)\in\mathbb{R}^2}  \underbrace{m(\xi_1,\xi_2)\left( i n_1\right)}_{=:m_1(n_1,\xi_1,\xi_2)} \widehat{\Psi}_{n_1}(\xi_1)\widehat{\partial_2\theta}_{n_2}(\xi_2)\widehat{\theta}_{n_1+n_2}(\xi_1+\xi_2)d\xi_1d\xi_2.
\end{align}
Then the symbol $m_1$ can be bounded as
\begin{align*}
|m_1(n_1,\xi_1,\xi_2)| &=\left(|\xi_1+\xi_2|^k -k|\xi_1||\xi_2|^{k-1}-|\xi_2|^k\right)|\xi_1+\xi_2|^k |n_1|\\
&\le_C \left(|\xi_1|^k + |\xi_1|^2|\xi_2|^{k-2} \right)|\xi_1+\xi_2|^k |n_1|\\
&\le_C \left(|\xi_1|^k + |\xi_1|^2|\xi_2|^{k-2} \right) |\xi_1+\xi_2|^k|n_1+n_2| + \left(|\xi_1|^k + |\xi_1|^2|\xi_2|^{k-2} \right) |\xi_1+\xi_2|^k|n_2|.
\end{align*}
Since it holds that, for $(\xi_1,\xi_2)\in\mathbb{R}^2$,
\[
\left(|\xi_1|^k + |\xi_1|^2|\xi_2|^{k-2} \right) |\xi_1+\xi_2|^k\le_C  |\xi_1|^{k+1}|\xi_1+\xi_2|^{k-1} + |\xi_1|^2|\xi_2|^{k-1}|\xi_1+\xi_2|^{k-1},
\]
We have
\begin{align*}
|m_{1}(n_1,\xi_1,\xi_2)|&\le_C |\xi_1|^{k+1}|\xi_1+\xi_2|^{k-1}|n_1+n_2|\\
&\  +|\xi_1|^2|\xi_2|^{k-1}|\xi_1+\xi_2|^{k-1}|n_1+n_2|\\
& \ + |\xi_1|^k|\xi_1+\xi_2|^k|n_2|\\
& \ + |\xi_1|^2|\xi_2|^{k-2}|\xi_1+\xi_2|^k|n_2|\\
& =: m_{11}+m_{12}+m_{13}+m_{14}.
\end{align*}
Hence, plugging this into \eqref{I21_stevie_wonder}, we have
\begin{align}\label{other_terms}
|I_{1}|\le C \sum_{i=1}^4 I_{1i},\quad I_{1i}:= \sum_{n_1,n_2\in\mathbb{Z}}\int_{(\xi_1,\xi_2)\in\mathbb{R}^2} m_{2i}(n_1,n_2,\xi_1,\xi_2)\left|\widehat{\Psi}_{n_1}(\xi_1)\widehat{\partial_2\theta}_{n_2}(\xi_2)\widehat{\theta}_{n_1+n_2}(\xi_1+\xi_2)\right|d\xi_1d\xi_2.
\end{align}
Using the notations in \eqref{fourier_convolution_no}, we can write
\begin{align}
I_{11}&=\Big\rVert \widehat{|D_2|^{k+1}\Psi}\left(\widehat{\partial_2\theta}*\widehat{|D_2|^{k-1}|D_1|\theta} \right)\Big\rVert_{L^1}\label{211_in1211}\\
I_{12}&=\Big\rVert \widehat{|D_2|^{2}\Psi}\left(\widehat{|D_2|^{k-1}\partial_2\theta}*\widehat{|D_2|^{k-1}|D_1|\theta} \right)\Big\rVert_{L^1}\label{211_in1212}\\
I_{13}&=\Big\rVert \widehat{|D_2|^{k}\Psi}\left(\widehat{|D_1|\partial_2\theta}*\widehat{|D_2|^{k}\theta} \right)\Big\rVert_{L^1}\label{211_in1213}\\
I_{14}&=\Big\rVert \widehat{|D_2|^{2}\Psi}\left(\widehat{|D_2|^{k-2}|D_1|\partial_2\theta}*\widehat{|D_2|^{k}\theta} \right)\Big\rVert_{L^1}\label{211_in1214}
\end{align}
 Using \eqref{young_Fourier_1} and \eqref{special_calse_1}, we see that 
\begin{align*}
|I_{11}|&\le \Big\rVert \widehat{|D_2|^{k+1}\Psi}\Big\rVert_{L^2}\rVert \widehat{\partial_2\theta}\rVert_{L^1}\Big\rVert \widehat{|D_2|^{k-1}|D_1|\theta}\Big\rVert_{L^2}\le_C \rVert \Psi\rVert_{H^{k+1}}^2\rVert \theta\rVert_{H^k}\le \rVert u\rVert_{H^k}^2\rVert \theta\rVert_{H^k},\\
|I_{12}|&\le\Big \rVert \widehat{|D_2|^2\Psi}\Big\rVert_{L^1}\Big\rVert \widehat{|D_2|^{k-1}\partial_2\theta}\Big\rVert_{L^2}\Big\rVert \widehat{|D_2|^{k-1}|D_1|\theta}\Big\rVert_{L^2}\le_C \rVert \Psi\rVert_{H^{k+1}}^2\rVert \theta\rVert_{H^k}\le \rVert u\rVert_{H^k}^2\rVert \theta\rVert_{H^k},\\
|I_{14}|&\le \Big\rVert \widehat{|D_2|^2\Psi}\Big\rVert_{L^1}\Big\rVert \widehat{|D_2|^{k-2}|D_1|\partial_2\theta}\Big\rVert_{L^2}\Big\rVert \widehat{|D_2|^k\theta}\Big\rVert_{L^2}\le_C \rVert \Psi\rVert_{H^{k+1}}^2\rVert \theta\rVert_{H^k}\le \rVert u\rVert_{H^k}^2\rVert \theta\rVert_{H^k},
\end{align*}
thus
\begin{align}\label{rest_estimates_1dx}
|I_{11}|+|I_{12}|+|I_{14}|\le C\rVert u\rVert_{H^k}^2\rVert \theta\rVert_{H^k}.
\end{align}
To estimate $I_{13}$, we apply \textcolor{black}{the} H\"older inequality, yielding that
\begin{align}\label{I213estimate}
|I_{13}|& \le \Big\rVert \widehat{|D_2|^k\Psi}\Big\rVert_{L^p}\Big\rVert \widehat{|D_1|\partial_2\theta}*\widehat{|D_2|^k\theta}\Big\rVert_{L^q}\quad \text{ (for $p^{-1}+q^{-1}=1$)}\nonumber\\
& \le \Big\rVert \widehat{|D_2|^k\Psi}\Big\rVert_{L^p} \Big\rVert \widehat{|D_1|\partial_2\theta}\Big\rVert_{L^{\frac{2q}{q+2}}}\Big\rVert \widehat{|D_2|^k\theta}\Big\rVert_{L^2}.
\end{align}
Fixing $q$ so that
\begin{align*}
 \max\left\{\frac{3}{k-2},2\right\}<q \in (2,\infty),
\end{align*}
we have  
\begin{align*}
1<p<2,\quad 1<\frac{2q}{2+q}<2.
\end{align*} 
Applying \eqref{Hoilder_fouydirsd} to $\Big\rVert \widehat{|D_2|^k\Psi}\Big\rVert_{L^p}$, we get
\[
\Big\rVert \widehat{|D_2|^k\Psi}\Big\rVert_{L^p} \le C \rVert |D_2|^k\Psi\rVert_{H^1}\le C\rVert u\rVert_{H^{k}}.
\]
\textcolor{black}{Similarly,} we apply \eqref{Hoilder_fouydirsd} to $\Big\rVert \widehat{|D_1|\partial_2\theta}\Big\rVert_{L^{\frac{2q}{q+2}}}$, obtaining
\[
\Big\rVert \widehat{|D_1|\partial_2\theta}\Big\rVert_{L^{\frac{2q}{q+2}}}\le C\Big\rVert |D_1|\partial_2\theta\Big\rVert_{H^{k-2}}\le C\rVert u\rVert_{H^k}.
\]
Therefore, \eqref{I213estimate} reads
\[
|I_{13}|\le C\rVert u\rVert_{H^k}^2\rVert \theta\rVert_{H^k}.
\]
\textcolor{black}{Plugging this estimate and \eqref{rest_estimates_1dx} into \eqref{other_terms} gives}
\begin{align*}
|I_{1}|\le C \rVert u\rVert_{H^k}^2\rVert \theta\rVert_{H^k}. 
\end{align*}
Together with \eqref{I22estimate} and \eqref{I2split}, we obtain \eqref{non_lin_es_1}.

\textbf{Proof of \eqref{non_lin_es_2}.}  We have
  \begin{align*}
  J&:= \int \partial_1u_2\partial_2\theta|D| \left(\partial_{111}|D|^{-1}\theta\right)dx = \int |D|\left(\partial_1u_2 \partial_2\theta \right)\left(\partial_{111}|D|^{-1}\theta \right)dx.
  \end{align*}
  Hence, we \textcolor{black}{also} have
  \begin{align}\label{B_21How}
  |J|\le \Big\rVert \partial_{111}|D|^{-1}\theta\Big\rVert_{L^2}\Big\rVert |D|(\partial_1u_2\partial_2\theta)\Big\rVert_{L^2}.
   \end{align}
 It is straightforward that
 \begin{align}\label{Jeff}
 \Big\rVert \partial_{111}|D|^{-1}\theta\Big\rVert_{L^2}=\rVert \partial_{11}\nabla \Psi\rVert_{L^2}=\rVert \nabla\partial_1u_2\rVert_{L^2}=\Big\rVert \partial_1|D|^{-1}u_2\Big\rVert_{H^2}.
 \end{align}
 To estimate $\Big\rVert |D|(\partial_1u_2\partial_2\theta)\Big\rVert_{L^2}$, we look at the Fourier coefficients. Using the notations in \eqref{fourier_convolution_no} and the Plancherel theorem, we can \textcolor{black}{bound} it as 
 \begin{align}\label{D1sd_parbno}
 \Big\rVert |D|(\partial_1u_2\partial_2\theta)\Big\rVert_{L^2} \le \rVert (|\xi|+|n|)\widehat{\partial_1u_2}*\widehat{\partial_2\theta}\rVert_{L^2}=:\rVert \widehat{g}\rVert_{L^2}.
 \end{align}
  where $\widehat{g}$ is given by
  \[
  \widehat{g}_n(\xi)= (|\xi| + |n|)\sum_{n_1\in\mathbb{Z}}\int_{\mathbb{R}}|\widehat{\partial_1u_2}_{n-n_1}(\xi-\xi_1)||\widehat{\partial_2\theta}_{n_1}(\xi_1)|d\xi_1.
  \]
We estimate $\widehat{g}_n(\xi)$ as
\begin{align*}
|\widehat{g}_n(\xi)| &\le \sum_{n_1\in\mathbb{Z}}\int_{\mathbb{R}} (|\xi-\xi_1|+|n-n_1|)|\widehat{\partial_1u_2}_{n-n_1}(\xi-\xi_1)||\widehat{\partial_2\theta}_{n_1}(\xi_1)|d\xi_1\\
& \ +\sum_{n_1\in\mathbb{Z}}\int_{\mathbb{R}} \left|\widehat{\partial_1u_2}_{n-n_1}(\xi-\xi_1)\right|(|\xi_1| + |n_1|)\left|\widehat{\partial_2\theta}_{n_1}(\xi_1)\right|d\xi_1\\
&=\left(\widehat{|D|\partial_1 u_2}*\widehat{\partial_2\theta}\right)_n(\xi) + \left(\widehat{\partial_1 u_2}*\widehat{|D|\partial_2\theta} \right)_n(\xi).
\end{align*}
 Therefore, the Plancherel theorem and Young's inequality for convolution give us
 \begin{align}\label{choice_pq_g}
 \rVert \widehat{g}\rVert_{L^2}&\le \Big\rVert |D|\partial_1u_2\Big\rVert_{L^2}\rVert \widehat{\partial_2\theta}\rVert_{L^1} + \rVert \widehat{\partial_1u_2}\rVert_{L^p}\Big\rVert \widehat{|D|\partial_2\theta}\Big\rVert_{L^q}\text{ for $p^{-1}+q^{-1}=\frac{3}{2}$}\nonumber\\
 & \le  \Big\rVert \partial_1|D|^{-1}u_2\Big\rVert_{H^2}\rVert \theta\rVert_{H^k}+ \rVert \widehat{\partial_1u_2}\rVert_{L^p}\Big\rVert \widehat{|D|\partial_2\theta}\Big\rVert_{L^q},
 \end{align}
 where the last inequality follows from \eqref{special_calse_1}.
Let us estimate $\rVert \widehat{\partial_1u_2}\rVert_{L^p}$ and $\Big\rVert \widehat{|D|\partial_2\theta}\Big\rVert_{L^q}$ for appropriately chosen $p,q$. We choose \textcolor{black}{$q$ to satisfy $q=\frac{k}{k-1} $ so that $p=\frac{2k}{k+2}$.} Since $k>2$, it follows easily that
 \textcolor{black}{\begin{align*}
 1<p<2 \quad \mbox{and}\quad 1<q< 2.
 \end{align*}}
 Hence, applying \eqref{Hoilder_fouydirsd}  to $\rVert \widehat{\partial_1u_2}\rVert_{L^p}$ gives us
 \[
 \rVert \widehat{\partial_1u_2}\rVert_{L^p}\le C\rVert \partial_1u_2\rVert_{H^1}\le C\Big\rVert \partial_1 |D|^{-1}u_2\Big\rVert_{H^2}.
 \]
 Also, applying \eqref{Hoilder_fouydirsd}  to $\Big\rVert \widehat{|D|\partial_2\theta}\Big\rVert_{L^q}$, we get
 \[
 \Big\rVert \widehat{|D|\partial_2\theta}\Big\rVert_{L^q}\le C \Big\rVert |D|\partial_2\theta\Big\rVert_{H^{k-2}}\le C \rVert \theta\rVert_{H^k}.
 \]
 Thus \eqref{choice_pq_g} reads $\rVert \widehat{g}\rVert_{L^2}\le \Big\rVert \partial_1|D|^{-1}u_2\Big\rVert_{H^2}\rVert \theta\rVert_{H^k}$. Plugging this estimate and \eqref{Jeff} into \eqref{B_21How}, we arrive at
 \[
 |J|\le \Big\rVert \partial_1|D|^{-1}u_2\Big\rVert_{H^2}^2\rVert \theta\rVert_{H^k},
 \]
 which is the desired estimate.   \end{proof}
 
 Now, we give an estimate for $\rVert \theta\rVert_{H^k}$.
   \begin{proposition}\label{energy_IPM_estimate}
   Let $\theta\in C([0,T); H^k)$ be a solution to \eqref{IPM_theta}. Let $\gamma$ be as in \eqref{alpha_def}. Then there exist $\delta_0=\delta_0(\gamma)>0$  and $C=C(\gamma)$ such that if 
   \[
   \sup_{t\in[0,T]}\rVert \theta(t)\rVert_{H^k}<\delta_0,
   \] then
   \begin{equation}
       \label{eq:eshn} \begin{aligned}
\frac{1}2\frac{d}{dt}\rVert \theta\rVert_{{H}^k}^2\le -C\rVert u\rVert_{H^k}^2 + C\left( \rVert \nabla u_2\rVert_{L^\infty}\rVert \theta\rVert_{H^k}^2 + \rVert u\rVert_{L^2}^2\right).
\end{aligned}
   \end{equation}  
for $0\le t \le T$.
   \end{proposition}
\begin{proof}
In what follows  $|D_i|$ will denote either $|D_1|$ or $|D_2|$, which is defined in \eqref{d_denf1}.
Using \eqref{IPM_theta}, we compute
\begin{align}\label{energe1}
\frac{1}2\frac{d}{dt}\Big\rVert |D_i|^k \theta\Big\rVert_{L^2}^2 = -\int |D_i|^k(u\cdot\nabla \theta) |D_i|^k\theta dx +\int |D_i|^k\left(-\partial_2\rho_su_2\right)|D_i|^k\theta dx
\end{align}
We first simplify the second term \textcolor{black}{on} the right-hand side. Using the stream function $\Psi$ such that  $u_2=\partial_1\Psi$ and $\partial_1\theta=-\Delta \Psi$, we have
\begin{align}\label{linear_gets_ugly}
\int |D_i|^k(-\partial_2\rho_su_2)|D_i|^k\theta dx&=\int \partial_1|D_i|^k(-\partial_2\rho_s\Psi) |D_i|^k\theta dx = -\int |D_i|^k(-\partial_2\rho_s\Psi )|D_i|^k\partial_1\theta dx\nonumber\\
& = \int |D_i|^k(-\partial_2\rho_s\Psi) |D_i|^k\Delta \Psi dx=- \int\nabla |D_i|^k(-\partial_2\rho_s\Psi)\cdot\nabla|D_i|^k\Psi dx\nonumber\\
&= -  \int \partial_2\rho_s|\nabla |D_i|^k\Psi|^2dx\nonumber\\
& \  + \int \left(\nabla |D_i|^k(\partial_2\rho_s \Psi) - \partial_2\rho_s \nabla|D_i|^k\Psi\right)\cdot \nabla |D_i|^k\Psi dx\nonumber\\
&\le -C\Big\rVert |D_i|^ku\Big\rVert_{L^2}^2 +\rVert \Psi\rVert_{\dot{H}^{k+1}}\Big\rVert \nabla |D_i|^k(\partial_2\rho_s \Psi) - \partial_2\rho_s \nabla|D_i|^k\Psi\Big\rVert_{L^2},
\end{align}
where the last inequality is due to $-\partial_2\rho_s>\gamma$ and the Cauchy--Schwarz inequality.
A crude estimate  gives us
\[
\Big\rVert \nabla |D_i|^k(\partial_2\rho_s \Psi) - \partial_2\rho_s \nabla|D_i|^k\Psi\Big\rVert_{L^2}\le\footnote{\textcolor{black}{Technically speaking, this inequality does not follow from \eqref{standard_tamepp}. However, a rigorous proof of this inequality is  straightforward {\color{black} and we omit it.}}}  
C\rVert\partial_2\rho_s\rVert_{C^{k+1}}\rVert\Psi\rVert_{H^k}\le C\rVert \Psi\rVert_{H^k}\le  C_\eta \rVert\Psi\rVert_{L^2} + \eta\rVert\Psi\rVert_{\dot{H}^{k+1}},
\]
for any $\eta>0$, where the last inequality is due to the usual Sobolev interpolation theorem. Plugging this estimates into \eqref{linear_gets_ugly} with sufficiently small $\eta$, we obtain
\[
\sum_{i=1}^2\int |D_i|^k(-\partial_2\rho_su_2)|D_i|^k\theta dx\le -C\rVert u\rVert_{H^k}^2 + C\rVert \Psi\rVert_{L^2}^2.
\]
Furthermore, denoting
\[
G(x_2):=\int_{\mathbb{T}}\Psi(x_1,x_2)dx_1,
\]
we notice from \eqref{stream_IPM_2} that $g$ solves
\[
\partial_{22}G=0 \text{ and }\lim_{|x_2|\to \infty}G(x_2)=0.
\]
Hence, $G(x_2)=0$ for every $x_2$. \textcolor{black}{Then} the Poincar\'e inequality gives  us
\[
\rVert \Psi\rVert_{L^2}\le \rVert \partial_1\Psi\rVert_{L^2}\le \rVert u\rVert_{L^2}.
\]
Therefore we arrive at
\begin{align}\label{linear_full}
\sum_{i=1}^2\int |D_i|^k(-\partial_2\rho_su_2)|D_i|^k\theta dx\le -C\rVert u\rVert_{H^{k}}^2 + C\rVert u\rVert_{L^2}^2.
\end{align}
Note that in the derivation of this estimate, we have not used the condition that $k>2$. Indeed, repeating the same computations with $k=0$, we obtain
\begin{align}\label{L_2tesxc}
\frac{1}{2}\frac{d}{dt}\rVert \theta\rVert_{L^2}^2 = \int -\partial_2\rho_su_2\theta dx\le C \rVert u\rVert_{L^2}^2.
\end{align}
Now, we move on to estimate the first term \textcolor{black}{on} the right-hand side of \eqref{energe1}. As usual, we use the incompressibility of the velocity to obtain
\begin{align}\label{nonlinear_1}
\int |D_i|^k(u\cdot\nabla \theta) |D_i|^k\theta dx &= \int \left(|D_i|^k(u\cdot\nabla \theta) - u\cdot\nabla |D_i|^k\theta\right) |D_i|^k\theta dx = I.
\end{align}
 We claim that \begin{align}\label{once_show}
 |I|\le C\left( \rVert \nabla u_2\rVert_{L^\infty}\rVert \theta\rVert_{H^k}^2 +  \rVert u\rVert_{H^k}^2\rVert \theta\rVert_{H^k}\right), \text{ whether $i=1,2$.}
 \end{align}
Once the claim is \textcolor{black}{verified}, plugging it and \eqref{linear_full} into \eqref{energe1} gives us
\[
\frac{1}2\frac{d}{dt}\rVert \theta\rVert_{\dot{H}^k}^2 \le -C(1-C\rVert \theta\rVert_{H^k})\rVert u\rVert_{H^k}^2 +C\left( \rVert \nabla u_2\rVert_{L^\infty}\rVert \theta\rVert_{H^k}^2 + \rVert u\rVert_{L^2}^2\right).
\]
Hence, if $\rVert \theta(t)\rVert_{H^k}$ is sufficiently small depending on the implicit constant $C$, we obtain 
\[
\frac{d}{dt}\rVert \theta\rVert_{\dot{H}^k}^2 \le -C\rVert u\rVert_{H^k}^2 +C\left( \rVert \nabla u_2\rVert_{L^\infty}\rVert \theta\rVert_{H^k}^2 + \rVert u\rVert_{L^2}^2\right).
\]
Combining this with \eqref{L_2tesxc}, we derive \eqref{eq:eshn}.

In order to prove the claim, we split the cases when $i=1$ and $i=2$.

\textbf{When $i=1$.} In this case, \eqref{nonlinear_1} reads
\begin{align}\label{i1est}
\left| I \right| \le \Big\rVert |D_1|^k\theta\Big\rVert_{L^2}\Big\rVert |D_1|^k(u\cdot\nabla \theta) - u\cdot\nabla |D_1|^k\theta\Big\rVert_{L^2}.
\end{align}
The usual commutator estimate gives
\[
\Big\rVert |D_1|^k(u\cdot\nabla \theta) - u\cdot |D_1|^k\nabla\theta\Big\rVert_{L^2}\le_C\rVert u\rVert_{H^k} \rVert\nabla\theta\rVert_{L^\infty} + \rVert u\rVert_{W^{1,\infty}}\rVert\nabla \theta\rVert_{H^{k-1}}\le C\rVert u\rVert_{H^{k}}\rVert\theta\rVert_{H^k},
\]
where the last inequality follows from the Sobolev embedding $W^{1,\infty}(\Omega)\hookrightarrow H^k(\Omega)$. Also, note that
\[
\Big\rVert |D_1|^k\theta\Big\rVert_{L^2}=\Big\rVert |D_1|^{k-2}\partial_{11}\theta\Big\rVert_{L^2}=\Big\rVert |D_1|^{k-2}\partial_1\Delta\Psi\Big\rVert\le \rVert \Psi\rVert_{H^{k+1}}\le \rVert u\rVert_{H^k}.
\]
Therefore, the estimate \eqref{i1est} becomes
\begin{align}\label{i1est2}
|I|\le C\rVert u\rVert_{H^k}^2\rVert \theta\rVert_{H^k},\text{ when $i=1$.}
\end{align}

\textbf{When $i=2$.} In this case, \eqref{nonlinear_1} reads
\begin{align}\label{i2est}
I&= \int\left(|D_2|^k(u_1\partial_1 \theta) - u_1 |D_2|^k\partial_1\theta\right) |D_2|^k\theta dx+  \int\left(|D_2|^k(u_2\partial_2 \theta) - u_2 |D_2|^k\partial_2\theta\right) |D_2|^k\theta dx\nonumber\\
&=: I_{1}+I_2.
\end{align}
Again, the usual commutator estimate and the Cauchy--Schwarz inequality give us
\begin{align}\label{I1est}
|I_1|&\le_C \left(\rVert u\rVert_{H^k}\rVert\partial_1\theta\rVert_{L^\infty}+\rVert u\rVert_{W^{1,\infty}}\rVert \partial_1\theta\rVert_{H^{k-1}} \right)\rVert \theta\rVert_{H^k}\nonumber\\
&\le_C \rVert u\rVert_{H^k}(\rVert \partial_1\theta\rVert_{L^\infty}+\rVert u\rVert_{W^{1,\infty}})\rVert \theta\rVert_{H^k}\nonumber\\
& \le_C \rVert u\rVert_{H^k}\rVert \Psi\rVert_{W^{2,\infty}}\rVert \theta\rVert_{H^k}\nonumber\\
& \le_C \rVert u\rVert_{H^k}\rVert \Psi\rVert_{H^{k+1}}\rVert \theta\rVert_{H^k}\nonumber\\
&\le_C \rVert u\rVert_{H^k}^2\rVert \theta\rVert_{H^k}.
\end{align}
 A necessary estimate for $I_2$ was already derived in Lemma~\ref{Nonlinear_estimate_lem1}, namely, we have
\[
|I_2|\le_C \rVert \nabla u_2\rVert_{L^\infty}\rVert \theta\rVert_{H^k}^2 + \rVert u\rVert_{H^k}^2\rVert \theta\rVert_{H^k}.
\]
Finally, plugging this and \eqref{I1est} into \eqref{i2est}, we obtain
\[
|I|\le_C  \rVert \nabla u_2\rVert_{L^\infty}\rVert \theta\rVert_{H^k}^2 +  \rVert u\rVert_{H^k}^2\rVert \theta\rVert_{H^k} \text{ when $i=2$.}
\]
Collecting this and \eqref{i1est2}, we have proved \eqref{once_show}.
\end{proof}

\begin{proposition}\label{velocity_derivative}
Under the same assumption in Proposition~\ref{energy_IPM_estimate}, there exist $\delta_0=\delta_0(\gamma)>0$  and $C=C(\gamma)$ such that if 
\[
\sup_{t\in[0,T]}\rVert \theta(t)\rVert_{H^k}<\delta_0,
\] then
\begin{align}
\frac{d}{dt}\rVert \partial_{11}\theta\rVert_{L^2}^2 &\le  - C\Big\rVert \partial_1|D|^{-1}u_2\Big\rVert_{{H}^2}^2  +C\left( \rVert u\rVert_{H^k}^2 + \rVert \partial_2u_2\rVert_{L^\infty}\right)\rVert \partial_{11}\theta\rVert_{L^2}^2 + C \rVert u_2\rVert_{L^2}^2, \label{laplacian_1}\\
\frac{d}{dt} \rVert \partial_{12}\theta\rVert_{L^2}  & \le -C\rVert u_2\rVert_{H^2}^2 + C\left(\rVert u\rVert_{H^k}^2 + \rVert \partial_2u_2\rVert_{L^\infty}\right)\rVert \partial_{12}\theta\rVert_{L^2}^2\nonumber\\
& \ +C\left( \Big\rVert \partial_1|D|^{-1}u_2\Big\rVert_{H^2}\rVert\partial_{12}\theta\rVert_{L^2}\rVert \theta\rVert_{H^k} +   \Big\rVert \partial_1|D|^{-1}u_2\Big\rVert_{{H}^2}^2 +\rVert u\rVert_{L^2}^2\right). \label{laplacian_2} 
\end{align}
\end{proposition}
\begin{proof}
Let $\mathcal{D}$ denote either $\partial_{12}$ or $\partial_{11}$. Using \eqref{IPM_theta}, we compute
\begin{align}\label{energe123}
\frac{1}2\frac{d}{dt}\rVert\mathcal{D} \theta\rVert_{L^2}^2 = -\int\mathcal{D} (u\cdot\nabla \theta)\mathcal{D}\theta dx +\int \mathcal{D}\left(-\partial_2\rho_su_2\right)\mathcal{D}\theta dx.
\end{align}
The second term \textcolor{black}{on} the right-hand side can be estimated, using $u_2=\partial_1\Psi$, as
\begin{align}\label{no_sacrifices_elton_john}
\int \mathcal{D}\left(-\partial_2\rho_su_2\right)\mathcal{D}\theta dx &= \int \partial_1\mathcal{D}(-\partial_2\rho_s \Psi)\mathcal{D}\theta dx\nonumber\\
& =-\int \mathcal{D}(-\partial_2\rho_s \Psi)\mathcal{D} (-\Delta\Psi) dx\nonumber\\
& = -\int \nabla\mathcal{D}(-\partial_2\rho_s \Psi) \cdot \mathcal{D}\nabla \Psi dx\nonumber\\
& =  -\int (-\partial_2\rho_s) |\nabla \mathcal{D}\Psi|^2 dx + \int \left( \nabla \mathcal{D}(\partial_2\rho_s\Psi) - \partial_2\rho_s\nabla\mathcal{D}\Psi \right)\cdot \mathcal{D}\nabla \Psi dx\nonumber\\
&\le - C\int  |\nabla \mathcal{D}\Psi|^2 dx +\int \left( \nabla \mathcal{D}(\partial_2\rho_s\Psi) - \partial_2\rho_s\nabla\mathcal{D}\Psi \right)\cdot \mathcal{D}\nabla \Psi dx\nonumber \\
& \le - C\int  |\nabla \mathcal{D}\Psi|^2 dx +  C\Big\rVert \nabla \mathcal{D}(\partial_2\rho_s\Psi) - \partial_2\rho_s\nabla\mathcal{D}\Psi \Big\rVert_{L^2}^2\nonumber\\
&=:-A_1+A_2.
\end{align}
where the last inequality follows from the Cauchy--Schwarz inequality.
For the term $A_2$, it is straightforward to see that
\begin{align*}
\nabla \mathcal{D}(\partial_2\rho_s\Psi) - \partial_2\rho_s\nabla\mathcal{D}\Psi =\begin{cases} \colvec{0 \\ \partial_{22}\rho_s\partial_{11}\Psi} & \text{ if $\mathcal{D}=\partial_{11}$,}\\
\colvec{\partial_{22}\rho_s\partial_{11}\Psi \\ \partial_{222}\rho_s\partial_1\Psi +2\partial_{22}\rho_s\partial_{12}\Psi} & \text{ if $\mathcal{D}=\partial_{12}$.}
\end{cases}
\end{align*}
Using the Sobolev interpolation theorem and the estimates for $\partial_2\rho_s$ in \eqref{rkskd2sd}, we can find
\begin{align*}
\rVert \partial_{22}\rho_s\partial_{11}\Psi\rVert_{L^2}&\le_C\rVert\partial_{1}u_2\rVert_{L^2}=\Big\rVert \partial_1|D|^{-1}u_2\Big\rVert_{H^1}\le C_\eta\rVert u_2\rVert_{L^2} + \eta\Big\rVert \partial_1|D|^{-1}u_2\Big\rVert_{H^2},\text{ for $\eta>0$,}\\
\rVert \partial_{222}\rho_s\partial_1\Psi +2\partial_{22}\rho_s\partial_{12}\Psi\rVert_{L^2}&\le_C\rVert \partial_{1}\Psi\rVert_{H^1}\le \rVert u\rVert_{H^1}\le C_\eta \rVert u\rVert_{L^2} + \eta\rVert u\rVert_{H^2} \text{ for $\eta>0$}.
\end{align*}
Hence,
\begin{align}\label{rocket_man}
{|A_2|}\le  \begin{cases}
C_\eta\rVert u_2\rVert_{L^2}^2 + \eta\rVert \partial_1|D|^{-1}u_2\rVert_{H^2}^2 & \text{ if $\mathcal{D}=\partial_{11}$,}\\
C_\eta \rVert u\rVert_{L^2}^2 + \eta\rVert u\rVert_{H^2}^2  & \text{ if $\mathcal{D}=\partial_{12}$,}
\end{cases}
\end{align}
for any $\eta>0$.
On the other hand,  $A_1$ in \eqref{no_sacrifices_elton_john} can be estimated as \begin{align}\label{Daftpunk}
\rVert \nabla \mathcal{D}\Psi\rVert_{L^2}\ge \begin{cases}
\rVert \nabla \partial_{1}u_2\rVert_{L^2}\ge \Big\rVert \partial_1|D|^{-1}u_2\Big\rVert_{\dot{H}^2}& \text{ if $\mathcal{D}=\partial_{11}$,}\\
\rVert \nabla \partial_{2}u_2\rVert_{L^2} \ge \Big\rVert |D|^2 u_2\Big\rVert_{L^2}-\rVert \nabla \partial_{1}u_2\rVert_{L^2}\ge \rVert u_2\rVert_{H^2}-\Big\rVert \partial_1|D|^{-1}u_2\Big\rVert_{{H}^2} - \rVert u\rVert_{L^2}&\text{ if $\mathcal{D}=\partial_{12}$}\textcolor{black}{.}
\end{cases}
\end{align}
\textcolor{black}{This implies}
\begin{align}\label{A22Poo}
{A_1} \ge \begin{cases}
C\rVert \partial_1|D|^{-1}u_2\rVert_{{H}^2}^2& \text{ if $\mathcal{D}=\partial_{11}$,}\\
C\rVert u_2\rVert_{H^2}^2-C\left(\Big\rVert \partial_1|D|^{-1}u_2\Big\rVert_{{H}^2}^2 +\rVert u\rVert_{L^2}^2\right)&\text{ if $\mathcal{D}=\partial_{12}$}.
\end{cases}
\end{align}
Plugging this and \eqref{rocket_man} into \eqref{no_sacrifices_elton_john} with sufficiently small $\eta>0$, we obtain
\begin{align}\label{Amy_macdonald}
\int \mathcal{D}\left(-\partial_2\rho_su_2\right)\mathcal{D}\theta dx \le_C \begin{cases}
- \Big\rVert \partial_1|D|^{-1}u_2\Big\rVert_{{H}^2}^2 + C\rVert u_2\rVert_{L^2}^2 & \text{ if $\mathcal{D}=\partial_{11}$,}\\
 -\rVert u_2\rVert_{H^2}^2 + C\left(\Big\rVert \partial_1|D|^{-1}u_2\Big\rVert_{{H}^2}^2 +\rVert u\rVert_{L^2}^2\right) &\text{ if $\mathcal{D}=\partial_{12}$,}
\end{cases}
\end{align}
which gives upper bounds for the last integral in \eqref{energe123}.

  The nonlinear term in \eqref{energe123} can be written as 
\[
\int\mathcal{D} (u\cdot\nabla \theta)\mathcal{D}\theta dx = \int\left(\mathcal{D} (u\cdot\nabla \theta)- u\nabla \mathcal{D}\theta\right)\mathcal{D}\theta dx,
\]
due to the incompressibility of $u$.
It is \textcolor{black}{clear }that
\[
\mathcal{D} (u\cdot\nabla \theta)- u\nabla \mathcal{D}\theta = \left(\mathcal{D}(u_1\partial_1\theta) - u_1\mathcal{D}\partial_1\theta\right) + \left(\mathcal{D}(u_2\partial_2\theta) - u_2\mathcal{D}\partial_2\theta\right),
\]
\textcolor{black}{thus}
\begin{align}\label{D_B12_estimate_Wann_ist_Uhr}
\int\mathcal{D} (u\cdot\nabla \theta)\mathcal{D}\theta dx  = B_1+B_2,
\end{align}
where
\begin{align}\label{Georgia_mordo_mynameis}
B_1:=\int \left(\mathcal{D}(u_1\partial_1\theta) - u_1\mathcal{D}\partial_1\theta\right)\mathcal{D}\theta dx,\quad  B_2:=\int \left(\mathcal{D}(u_2\partial_2\theta) - u_2\mathcal{D}\partial_2\theta\right)\mathcal{D}\theta dx.
\end{align}
 We will estimate $B_1$ and $B_2$ when $\mathcal{D}=\partial_{11}$ or $\partial_{12}$ separately.
 
 \textbf{Estimate for $B_1$ when $\mathcal{D}=\partial_{11}$.}
In this case, $B_1$ reads 
\[
B_1 = \int (\partial_{11}u_1\partial_1\theta + 2\partial_1u_1\partial_{11}\theta)\partial_{11}\theta dx.
\]
Hence the Cauchy--Schwarz inequality gives us
\begin{align}
|B_1|&\le \rVert \partial_{11}\theta\rVert_{L^2}\left(\rVert \partial_{11}u_1\rVert_{L^2}\rVert \partial_1\theta\rVert_{L^\infty} +\rVert \partial_1u_1\rVert_{L^\infty}\rVert\partial_{11}\theta\rVert_{L^2} \right)\nonumber\\
&\le \eta \rVert \partial_{11}u_1\rVert_{L^2}^2 + C_\eta \left(\rVert \partial_{11}\theta\rVert_{L^2}\rVert\partial_1\theta\rVert_{L^\infty} \right)^2 + \rVert \partial_1u_1\rVert_{L^\infty}\rVert \partial_{11}\theta\rVert_{L^2}^2 \text{ for any $\eta>0$}.\label{emotion_1}
\end{align}
Using $u=\nabla^\perp\Psi$ and the Sobolev embedding theorem gives us
\begin{align}
\rVert \partial_{11}u_1\rVert_{L^2}&=\rVert \partial_{12}u_2\rVert_{L^2}\le_C \rVert \partial_1|D|^{-1}u_2\rVert_{H^2},\label{estimate_nonlinear_1}\\
\rVert\partial_1\theta\rVert_{L^\infty}&=\rVert\Delta \Psi\rVert_{L^\infty}\le_C \rVert \Delta\Psi\rVert_{H^{k-1}}=\rVert \nabla \Psi\rVert_{H^k}=\rVert u\rVert_{H^k},\label{estimate_nonlinear_2}\\
\rVert \partial_1u_1\rVert_{L^\infty}&=\rVert \partial_2u_2\rVert_{L^\infty}\label{estimate_nonlinear_3}
\end{align}
Therefore, $B_1$ in \eqref{emotion_1} can be estimated, again by using the Cauchy--Schwarz inequality,
\begin{align}\label{My_name_is_Giovanni_Georgio}
|B_1|\le \eta \Big\rVert \partial_1|D|^{-1}u_2\Big\rVert_{H^2}^2 +\left(C_\eta \rVert u\rVert_{H^k}^2 + C\rVert \partial_2u_2\rVert_{L^\infty}\right)\rVert \partial_{11}\theta\rVert_{L^2}^2,\text{ for $\eta>0$, $\mathcal{D}=\partial_{11}$.}
\end{align}

 \textbf{Estimate for $B_1$ when $\mathcal{D}=\partial_{12}$.}
 When $\mathcal{D}=\partial_{12}$, $B_1$ in \eqref{Georgia_mordo_mynameis} reads
  \[
 B_1 = \int \partial_{12}\theta\left(\partial_{12}u_1\partial_1\theta + \partial_1u_1\partial_{12}\theta+\partial_2u_1\partial_{11}\theta\right)dx.
 \]
Hence, the Cauchy--Schwarz inequality yields
\begin{align}\label{Wana}
|B_1|&\le \rVert\partial_{12}\theta\rVert_{L^2}\left(\rVert \partial_{12}u_1\rVert_{L^2}\rVert\partial_1\theta\rVert_{L^\infty}+\rVert \partial_{1}u_1\rVert_{L^\infty}\rVert \partial_{12}\theta\rVert_{L^2}+\rVert \partial_2u_1\rVert_{L^\infty}\rVert\partial_{11}\theta\rVert_{L^2}\right).
\end{align}
As before,  $u=\nabla^\perp\Psi$ and the Sobolev embedding theorem tell us
\begin{align}
\rVert \partial_{12}u_1\rVert_{L^2}&=\rVert \partial_{22}u_2\rVert_{L^2}=\rVert u_2\rVert_{H^2},\label{bused_1}\\
\rVert \partial_2u_1\rVert_{L^\infty}&\le_C \rVert u\rVert_{H^k}\label{bused_4}\\
 \rVert \partial_{11}\theta\rVert_{L^2}&=\rVert \partial_1\Delta\Psi\rVert_{L^2}\le\rVert u_2\rVert_{H^2},\label{bused_5}
\end{align}
Plugging these estimates with \eqref{estimate_nonlinear_2} and \eqref{estimate_nonlinear_3} into \eqref{Wana}, we get
\begin{align*}
|B_{1}|&\le_C \rVert \partial_{12}\theta\rVert_{L^2}\left(\rVert u_2\rVert_{H^2}\rVert u\rVert_{H^k} + \rVert \partial_2u_2\rVert_{L^\infty}\rVert\partial_{12}\theta\rVert_{L^2}\right).
\end{align*}
Hence, applying the Cauchy--Schwarz inequality, we derive
\begin{align}\label{B_11estimate_3}
|B_1|\le \eta\rVert u_2\rVert_{H^2}^2 + \left(C_\eta\rVert u\rVert_{H^k}^2 + C\rVert \partial_2u_2\rVert_{L^\infty}\right)\rVert \partial_{12}\theta\rVert_{L^2}^2,\text{ for $\eta>0$ and $\mathcal{D}=\partial_{12}$}
\end{align}

 \textbf{Estimate for $B_2$ when $\mathcal{D}=\partial_{11}$.}
  When $\mathcal{D}=\partial_{11}$, $B_2$ in \eqref{Georgia_mordo_mynameis} reads
  \begin{align}\label{bsplite}
  B_2 &= \int \partial_{11}u_2\partial_2\theta\partial_{11}\theta dx  + 2\int \partial_1u_2\partial_{12}\theta\partial_{11}\theta dx\nonumber\\
  & =-\int \partial_1u_2\partial_2\theta \partial_{111}\theta dx  +  \int \partial_1 u_2 \partial_{12}\theta \partial_{11}\theta dx =: - B_{21}+B_{22} \end{align}
  We estimate $B_{22}$ first. It can be  estimated  as
  \begin{align*}
  |B_{22}|\le \rVert \partial_{11}\theta\rVert_{L^2}\rVert \partial_{1}u_2 \partial_{12}\theta\rVert_{L^2}\le \rVert \partial_{11}\theta\rVert_{L^2} \rVert \partial_1u_2\rVert_{L^p}\rVert \partial_{12}\theta\rVert_{L^q},\text{ for $p^{-1}+q^{-1}=1/2$.}
  \end{align*}
  We choose $q$ so that $\rVert \partial_{12}\theta\rVert_{L^q}$ can be bounded by $\rVert u\rVert_{H^k}$. To do so, we choose 
  \[
  q:=\begin{cases}
  \frac{2}{3-k}&\text{ if $k\le 5/2$,}\\
  4 & \text{ if $k> 5/2$.}
  \end{cases}
  \]
  Then, it is elementary to see that $\frac{3q-2}q\le k$. Hence, the usual Sobolev embedding gives us
  \begin{align*}
  \rVert \partial_{12}\theta\rVert_{L^q}=\rVert \partial_2\Delta\Psi\rVert_{L^q}\le_C\rVert u\rVert_{W^{2,q}}\le_C \rVert u\rVert_{H^{\frac{3q-2}q}}\le_C \rVert u\rVert_{H^k}.
  \end{align*}
  On the other hand, since $k>2$, it holds that $q>2$, hence $p<\infty$. Thus, again the Soboelv embedding yields
  \[
  \rVert \partial_1u_2\rVert_{L^p}=\Big\rVert \partial_1|D|^{-1} u_2\Big\rVert_{W^{1,p}}\le_C \Big\rVert \partial_1|D|^{-1}u_2\Big\rVert_{H^2}.
   \]
  Consequently, we obtain
  \begin{align}\label{mine_name_ist_Giovanni}
  |B_{22}|\le_C \rVert \partial_{11}\theta\rVert_{L^2}\rVert u\rVert_{H^k}\Big\rVert \partial_1|D|^{-1}u_2\Big\rVert_{H^2}\le \eta\Big\rVert \partial_1|D|^{-1}u_2\Big\rVert_{H^2}^2 + C_\eta \rVert u\rVert_{H^k}^2\rVert\partial_{11}\theta\rVert_{L^2}^2,\text{ for $\eta>0$.}
  \end{align}
  A necessary estimate for $B_{21}$ was already derived in Lemma~\ref{Nonlinear_estimate_lem1}, namely,
  \[
  |B_{21}|\le C\Big\rVert \partial_1|D|^{-1}u_2\Big\rVert_{H^2}^2\rVert\theta\rVert_{H^k}.
  \]
 Combining this with \eqref{mine_name_ist_Giovanni}, we prove that $B_2$ can be estimated as
 \begin{align}\label{B1espapr}
 |B_2|\le (\eta + C\rVert \theta\rVert_{H^k})\Big\rVert \partial_1|D|^{-1}u_2\Big\rVert_{H^2}^2 + C_\eta \rVert u\rVert_{H^k}^2\rVert \partial_{11}\theta\rVert_{L^2}^2,\text{ for $\eta>0$, $\mathcal{D}=\partial_{11}$.}
 \end{align}

 \textbf{Estimate for $B_2$ when $\mathcal{D}=\partial_{12}$.}
 In this case, $B_2$ in \eqref{Georgia_mordo_mynameis} can be written as
 \begin{align}\label{B_2grading}
 B_2= \int \left( \partial_{12}u_2\partial_2\theta + \partial_1u_2\partial_{22}\theta + \partial_2u_2\partial_{12}\theta \right)\partial_{12}\theta dx.
  \end{align}
The contribution from $\partial_{12}u_2$ can be further computed by integration by parts as
 \begin{align*}
 \int \partial_{12}u_2\partial_2\theta\partial_{12}\theta dx &= - \int \partial_2u_2 (\partial_{12}\theta)^2dx -\int \partial_{2}u_2\partial_2\theta \partial_{112}\theta dx\\
 & = - \int \partial_2u_2 (\partial_{12}\theta)^2dx + \int \partial_{22}u_2\partial_2\theta \partial_{11}\theta dx + \int \partial_2 u_2\partial_{22}\theta\partial_{11}\theta dx.
 \end{align*}
 \textcolor{black}{Inserting the above equation} into \eqref{B_2grading}, we have
 \[
 B_2 = \int \partial_1 u_2\partial_{22}\theta\partial_{12}\theta dx+\int \partial_{22} u_2\partial_{2}\theta\partial_{11}\theta dx+\int \partial_2 u_2\partial_{22}\theta\partial_{11}\theta dx=:B_{21}+B_{22}+B_{23}.
 \]
 We estimate $B_{22}$ first and then derive necessary upper bounds for $B_{21}$ and $B_{23}$. For $B_{22}$, we use \eqref{bused_5} and obtain
 \begin{align}\label{B22estimate_easy}
 |B_{22}|\le \rVert \partial_2\theta\rVert_{L^\infty}\rVert\partial_{22}u_2\rVert_{L^2}\rVert \partial_{11}\theta\rVert_{L^2}\le \rVert \partial_2\theta\rVert_{L^\infty}\rVert u_2\rVert_{H^2}^2 \le \rVert \theta\rVert_{H^k}\rVert u_2\rVert_{H^2}^2.
 \end{align}
 For $B_{23}$, again \eqref{bused_5} gives us
 \begin{align}\label{B_23halfway}
 |B_{23}|\le\rVert \partial_2u_2\partial_{22}\theta\rVert_{L^2}\rVert \partial_{11}\theta\rVert_{L^2}\le \rVert \partial_2u_2\partial_{22}\theta\rVert_{L^2}\rVert u_2\rVert_{H^2}\le \rVert \partial_2u_2\rVert_{L^p}\rVert \partial_{22}\theta\rVert_{L^q}\rVert u_2\rVert_{H^2},
 \end{align}
 where the last inequality follows from the H\"older inequality with the following choice of $p,q$:
 \begin{align}\label{pq_always_same}
   q:=\begin{cases}
  \frac{2}{3-k},&\text{ if $k\le 5/2$,}\\
  4 & \text{ if $k> 5/2$.}
  \end{cases},\quad p^{-1}+q^{-1}=1/2.
 \end{align}
 As we saw before, this choice of $p,q$ gives us
 \begin{align}\label{same_qchoice}
 \frac{3q-2}{q}\le k,\quad p<\infty.
 \end{align}
 Hence, the Sobolev embedding theorems tell us
 \begin{align}\label{recycle_partial_22}
 \rVert \partial_2u_2\rVert_{L^p}\le_C \rVert \partial_2u_2\rVert_{H^1}\le \rVert u_2\rVert_{H^2},\quad \rVert \partial_{22}\theta\rVert_{L^q}\le_C \rVert \theta\rVert_{W^{2,q}}\le_C \rVert \theta\rVert_{H^{\frac{3q-2}{q}}}\le_C\rVert \theta\rVert_{H^k}.
 \end{align}
 With these estimates, we read \eqref{B_23halfway} as
 \begin{align}\label{b32Hundert_eins}
 |B_{23}|\le C \rVert u_2\rVert_{H^2}^2 \rVert \theta\rVert_{H^k}.
 \end{align}
 For $B_{21}$, we estimate it in the same manner, that is, with the same choice of $p,q$ as in \eqref{pq_always_same},
 \begin{align}\label{B_21david_bowie_overrrated}
 |B_{21}|\le \rVert \partial_1u_2\rVert_{L^p}\rVert \partial_{22}\theta\rVert_{L^q}\rVert\partial_{12}\theta\rVert_{L^2}.
 \end{align}
 Again the Soboelv embedding theorems give us (since $p<\infty$)
 \begin{align*}
 \rVert \partial_1u_2\rVert_{L^p} &=\rVert \partial_1u_2\rVert_{H^1}\le \Big\rVert \partial_1|D|^{-1}u_2\Big\rVert_{H^2}.
 \end{align*}
Plugging this and the estimate for $\rVert\partial_{22}\theta\rVert_{L^q}$ in  \eqref{recycle_partial_22} into \eqref{B_21david_bowie_overrrated}, we obtain
\[
|B_{21}|\le \Big\rVert \partial_1|D|^{-1}u_2\Big\rVert_{H^2}\rVert\theta\rVert_{H^k}\rVert\partial_{12}\theta\rVert_{L^2}.
\]
Combining this with \eqref{b32Hundert_eins} and \eqref{B22estimate_easy}, we arrive at 
\begin{align}\label{last_B2}
|B_2|\le_C \rVert \theta\rVert_{H^k}\rVert u_2\rVert_{H^2}^2 + \Big\rVert \partial_1|D|^{-1}u_2\Big\rVert_{H^2}\rVert\partial_{12}\theta\rVert_{L^2}\rVert \theta\rVert_{H^k},\text{ when $\mathcal{D}=\partial_{12}$}.
\end{align}

 Now, we collect all the estimates. From \eqref{My_name_is_Giovanni_Georgio} and \eqref{B1espapr}, the nonlinear term \eqref{D_B12_estimate_Wann_ist_Uhr} can be bounded as 
 \begin{align}\label{nonlinear_123_emma}
&\left| \int\mathcal{D} (u\cdot\nabla \theta)\mathcal{D}\theta dx\right| \nonumber\\
 & \quad \le \left(C\rVert \theta\rVert_{H^k} + \eta \right)\Big\rVert \partial_1|D|^{-1}u_2\Big\rVert_{H^2}^2 +\left(C_\eta \rVert u\rVert_{H^k}^2 +C \rVert \partial_2u_2\rVert_{L^\infty}\right)\rVert \partial_{11}\theta\rVert_{L^2}^2,\text{ for $\eta>0$ when $\mathcal{D}=\partial_{11}$.}
 \end{align}
 Choosing $\eta$ sufficiently small and combining it with the estimate in \eqref{Amy_macdonald} when $\mathcal{D}=\partial_{11}$, 
 the equation in \eqref{energe123} gives \eqref{laplacian_1}.

 Similarly, from \eqref{B_11estimate_3} and \eqref{last_B2}, the nonlinear term \eqref{D_B12_estimate_Wann_ist_Uhr} can be bounded as 
 \begin{align}\label{nonlinear_1234_emma}
\left| \int\mathcal{D} (u\cdot\nabla \theta)\mathcal{D}\theta dx\right| &\le \left(C\rVert \theta\rVert_{H^k}+\eta\right)\rVert u_2\rVert_{H^2}^2 + \left(C_\eta\rVert u\rVert_{H^k}^2 + C\rVert \partial_2u_2\rVert_{L^\infty}\right)\rVert \partial_{12}\theta\rVert_{L^2}^2\nonumber\\
 & \ + C\Big\rVert \partial_1|D|^{-1}u_2\Big\rVert_{H^2}\rVert\partial_{12}\theta\rVert_{L^2}\rVert \theta\rVert_{H^k}\text{ for $\eta>0$ and $\mathcal{D}=\partial_{12}$}
 \end{align}
 Choosing $\eta$ sufficiently small and combining it with the estimate in \eqref{Amy_macdonald} when $\mathcal{D}=\partial_{12}$, 
 the equation in \eqref{energe123} gives \eqref{laplacian_2}.
\end{proof}

   \section{Energy structure analysis}
 Throughout the section, we will assume that  $\rho(t)$ is a smooth solution to the IPM equation satisfying
\begin{align}\label{assumption_delta}
 \rVert \rho(t)-\rho_s\rVert_{H^k}^2+  \int_0^T \rVert u(t)\rVert_{H^k}^2 dt=: \delta\le \delta_0, \text{ for $t\in[0,T]$ for some $T>1$ and $\delta_0\ll 1$.}
 \end{align}
 We denote
 \begin{align}
E(t):=\mathcal{E}(\rho(t)),
 \end{align}
 where $\mathcal{E}$ is  the potential energy defined in \eqref{potential_e}. 
 
 \begin{lemma}\label{energy_3x21}
 There exist $\delta_0>0$  such that if $\rho_0-\rho_s\in C^\infty_{c}(\Omega)$ and  \eqref{assumption_delta} holds, then 
 \[
 \frac{d}{dt}E(t)=-\rVert u(t)\rVert_{L^2}^2.
 \]
 \end{lemma}
 \begin{proof} 
 Using the definition of $\mathcal{E}$ in \eqref{potential_e}, we observe that for sufficiently large $s>0$,
 \begin{align*}
 \frac{d}{dt} \left( E_s(\rho)-E_s(\rho_0^*)\right) &= \int_{\left\{x\in \Omega : \ -s< \rho(t,x)<s\right\}}\rho_t(t,x)x_2dx\\
 & = \int_{\left\{x\in \Omega : \ -s< \rho(t,x)<s\right\}}-u(t,x)\cdot \nabla \left(\rho(t,x)-\rho_s\right) x_2 dx\\
 & \  + \int_{\left\{x\in \Omega : \ -s< \rho(t,x)<s\right\}}-u_2(t,x)\partial_2\rho_s x_2 dx\\
 & = \int_{\left\{x\in \Omega :\ -s< \rho(t,x)<s\right\}}u_2(t,x)\theta(t,x) dx - \int_{\left\{x\in \Omega :\ \pm s=\rho(t,x)\right\}}\theta (t,x)x_2 u(t,x) \cdot\vec{n}d\sigma(x)\\
& \  +  \int_{\left\{x\in \Omega : \ -s< \rho(t,x)<s\right\}}-u_2(t,x)\partial_2\rho_s x_2 dx\\
 &=: I_1(t) + I_2(t) + I_3(t),
 \end{align*}
 where $d\sigma$ denotes the surface measure and $\vec{n}$ is the outer normal vector on $\left\{x\in \Omega :\ \pm s=\rho(t,x)\right\}$. Let us estimate $I_1,I_2$ and $I_3$.
 
  For $I_2$,  it follows from \eqref{decay_estimate_123} that
  \begin{align*}
  |I_2(t)|&\le \rVert x_2u_2\rVert_{L^\infty}\sigma(\left\{x\in \Omega :\ \pm s=\rho(t,x)\right\})
\sup_{x\in \left\{x\in \Omega :\ \pm s=\rho(t,x)\right\}}|\theta(t,x)|\\
&\le\rVert x_2u_2\rVert_{L^\infty}\rVert x_2\theta\rVert_{L^\infty}\sup_{\left\{x\in \Omega :\ \pm s=\rho(t,x)\right\}}|x_2|^{-1}\sigma(\left\{x\in \Omega :\ \pm s=\rho(t,x)\right\})\\
&\le C(T)\sup_{\left\{x\in \Omega :\ \pm s=\rho(t,x)\right\}}|x_2|^{-1}\sigma(\left\{x\in \Omega :\ \pm s=\rho(t,x)\right\})
    \end{align*}
    Since $\rVert\rho-\rho_s\rVert_{C^1}\le  \rVert \rho-\rho_s\rVert_{H^k}\le \delta_0$ and $\rho_s$ is strictly decreasing, it holds that 
    \begin{align}\label{lifting_levelset_2}
    \inf_{\left\{x\in \Omega :\ \pm s=\rho(t,x)\right\}}|x_2| \to \infty \text{ as $s\to \infty$},
    \end{align}
   and the length of the level sets are bounded as well,
   \[
   \sigma(\left\{x\in \Omega :\ \pm s=\rho(t,x)\right\})< C(T)
   \]  for some constant $C(T)>0$. Therefore, 
    \begin{align}
    \lim_{s\to \infty}I_2(t) = 0 \text{ uniformly in $t\in [0,T]$.}
    \end{align}
    For $I_3$, again the decay \textcolor{black}{estimate} in \eqref{decay_estimate_123} give that
    \[
    |I_3(t)|\le \rVert u_2x_2\rVert_{L^\infty}\sup_{\left\{x\in \Omega :\ \pm s=\rho(t,x)\right\}}|\partial_2\rho_s(x)|\le C(T)\sup_{\left\{x\in \Omega :\ \pm s=\rho(t,x)\right\}}|\partial_2\rho_s(x_2)|.
    \]
    Since $\partial_2\rho_s\in H^k$, it holds that $\partial_2\rho_s(x_2)\to 0$ as $|x_2|\to \infty$. Thus, using \eqref{lifting_levelset_2}, we get
    \[
    \sup_{\left\{x\in \Omega :\ \pm s=\rho(t,x)\right\}}|\partial_2\rho_s(x_2)| \to 0 \text{ as $s\to \infty$, }
    \]
    yielding that 
    \[
    \lim_{s\to \infty}I_3(t)\to 0,\text{ uniformly in $t\in [0,T]$.}
    \]
 
 For $I_1$, since $u_2\theta $ is an integrable function, \textcolor{black}{there holds}
  \begin{align*}
 \lim_{s\to \infty}I_1(t) &=\lim_{s\to\infty}\int_{\left\{x\in \Omega :\ -s< \rho(t,x)<s\right\}}u_2(t,x)\theta(t,x) dx  = \int_{\Omega} u_2\theta  dx\\
 & =  - \int_{\Omega} \Psi \partial_1\theta dx\\
 & = \int_{\Omega}\Psi \Delta \Psi dx \\
 & = -\rVert u\rVert_{L^2}^2,
  \end{align*}
  and the limit is uniform in $t\in [0,T]$ thanks to the decay estimates in \eqref{decay_estimate_123}. Thus we arrive at
  \begin{align*}
  \frac{d}{dt}E(t)& = \lim_{s\to \infty} \frac{d}{dt}E_s(t) = \lim_{s\to\infty}I_1(t) = -\rVert u(t)\rVert_{L^2}^2.    \end{align*}
 \end{proof}

\begin{proposition}\label{energy_analysis}
There exist $\delta_0>0$  such that if $\rho_0-\rho_s\in C^\infty_{c}(\Omega)$ and  \eqref{assumption_delta} holds, then 
\begin{align}
E(t)&\le_C \frac{\delta}{t^k},\label{udecay1}\\
\frac{1}{t}\int_{t/2}^{t}\rVert u(s)\rVert_{L^2}^2 ds&\le_C \frac{\delta}{t^{k+1}}\label{udecay2}\textcolor{black}{,}
\end{align}
for all $t\in [0,T]$.
\end{proposition}
\begin{proof} Note that \eqref{u_to_ffstar} in Proposition~\ref{propoos} reads
\[
E(t)\le C\left(\rVert u(t)\rVert_{L^2}^{\frac{k-1}{k}}\rVert u\rVert_{H^k}^{\frac{1}{k}}\right)^2,
\]
which is equivalent to
\[
\rVert u(t)\rVert_{L^2}^2 \ge C E(t)^{\frac{k}{k-1}}\rVert u(t)\rVert_{H^k}^{\frac{2}{k-1}}.
\]
Hence, Lemma~\ref{energy_3x21} tells us
\[
\frac{d}{dt}E(t)\le -C E(t)^{\frac{k}{k-1}}\rVert u(t)\rVert_{H^k}^{\frac{2}{k-1}}.
\]
Applying Lemma~\ref{ABC_ODsdE} with $f(t):=E(t)$, $n=:\frac{k}{k-1}>1$, $\alpha:=\frac{1}{k-1}$ and $a(t):=C\rVert u(t)\rVert_{H^k}^{2}$, we obtain
\[
E(t)\le \frac{CA}{t^k},\text{ with $A:=\int_0^t \rVert u(t)\rVert_{H^k}^2\le \delta$,}
\]
where the bound for $A$ is due to \eqref{assumption_delta}. This proves \eqref{udecay1}.  Applying Lemma~\ref{ODEsdsdBDC} to $\frac{d}{dt}E(t)$, we also obtain \eqref{udecay2}.
\end{proof}

 In the rest of this section, we will derive decay rates of high Sobolev norms of the velocity.

\begin{proposition}\label{u_2high1}
There exist $\delta_0>0$  such that if $\rho_0-\rho_s\in C^\infty_c(\Omega)$ and \eqref{assumption_delta} holds, then 
\begin{align}
\frac{1}{t}\int_{t/2}^{t}\rVert u(s)\rVert_{H^2}^2+s\rVert u_2(s)\rVert_{H^2}^2 ds &\le_C\frac{\delta}{t^{k-1}}, \text{ for $t\in [0,T]$}.\label{udecay452}
\end{align}
\end{proposition}
\begin{proof}
Note that under the assumption~\eqref{assumption_delta}, Proposition~\ref{velocity_derivative} gives us
\begin{align}
\frac{d}{dt}\rVert \partial_{11}\theta\rVert_{L^2}^2 &\le  - C \Big\rVert \partial_1|D|^{-1}u_2\Big\rVert_{{H}^2}^2  +C\left(\rVert u\rVert_{H^k}^2 + \rVert \partial_2u_2\rVert_{L^\infty}\right)\rVert \partial_{11}\theta\rVert_{L^2}^2 + C\rVert u_2\rVert_{L^2}^2, \label{laplacian_11}\\
\frac{d}{dt} \rVert \partial_{12}\theta\rVert_{L^2}  & \le -C\rVert u_2\rVert_{H^2}^2 + C\left(\rVert u\rVert_{H^k}^2 + \rVert \partial_2u_2\rVert_{L^\infty}\right)\rVert \partial_{12}\theta\rVert_{L^2}^2\nonumber \\
& \ + C\left(\Big\rVert \partial_1|D|^{-1}u_2\Big\rVert_{H^2}\rVert\partial_{12}\theta\rVert_{L^2}\rVert \theta\rVert_{H^k} +   \Big\rVert \partial_1|D|^{-1}u_2\Big\rVert_{{H}^2}^2 +\rVert u\rVert_{L^2}^2\right). \label{laplacian_22} 
\end{align}
  
   Now, towards a contradiction, let us suppose there exists the first time $T_*>0$ such that 
   \begin{align}
\frac{1}{T_*}\int_{T_*/2}^{T_*}\rVert u(s)\rVert_{H^2}^2+s\rVert u_2(s)\rVert_{H^2}^2 ds &= \frac{M\delta}{T_*^{k-1}},\text{ for some  large $ 1\ll M\ll 1/\delta_0$,}\label{contradict3333ion3_1}
\end{align}
so that
\begin{align}
\frac{1}{t}\int_{t/2}^{t}\rVert u(s)\rVert_{H^2}^2+s\rVert u_2(s)\rVert_{H^2}^2 ds &\le \frac{M\delta}{t^{k-1}}, \text{ for $t\in [0,T_*]$.}\label{contradict3333ion3_12323}
\end{align}

 Towards a contradiction, we will first apply Lemma~\ref{diff_lem_time} to \eqref{laplacian_11}. To do so, let us denote 
  \begin{align*}
  H_1(t)&:=\rVert \partial_{11}\theta(t)\rVert_{L^2}^2,\\
   f_1(t)&:= C \rVert \partial_1|D|^{-1}u_2\rVert_{{H}^2}^2,\\
    h(t)&:= C\left(\rVert u(t)\rVert_{H^k}^2+\rVert \partial_2u_2(t)\rVert_{L^\infty}\right),\\
     g_1(t)&:=C\rVert u_2\rVert_{L^2}^2.
  \end{align*}
Then, \eqref{laplacian_11} reads
\begin{align}\label{laplacian_112sdxc2}
\frac{d}{dt}H_1(t)\le -f_1(t) + h(t)H_1(t) + g_1(t).
\end{align}
The decay rates for $H_1$ and $g_1$ can be computed as
\begin{align}
\frac{2}{t}\int_{t/2}^t H_1(s)ds& =\frac{2}{t}\int_{t/2}^t \rVert \partial_{11}\theta(s)\rVert_{L^2}ds = \frac{2}{t}\int_{t/2}^t\rVert \Delta u_2(s)\rVert_{L^2}ds\le \frac{M\delta}{t^k},\text{ for $t\in [0,T_*]$,}\label{time_av_appl1}\\
\frac{2}{t}\int_{t/2}^t g_1(s)ds& = \frac{2}{t}\int_{t/2}^t C\rVert u_2\rVert_{L^2}^2ds\le \frac{C\delta}{t^{k+1}},\text{ for $t\in [0,T_*]$,}\label{time_av_appl2}
\end{align}
where the last inequality in the first line is due to \eqref{contradict3333ion3_12323} and the second last inequality in the second line is from Proposition~\ref{energy_analysis}. Furthermore, for the integrability of $h$, we use the Sobolev embedding theorem to see that
\[
\rVert \partial_2u_2\rVert_{L^\infty}\le C_\eta\rVert u_2\rVert_{H^{2+\eta}}\le C_\eta \rVert u_2\rVert_{H^2}^{1-\frac{\eta}{k-2}}\rVert u_2\rVert_{H^{2+(k-2)}}^{\frac{\eta}{k-2}},\text{ for $0 < \eta \le  k-2$.}
\]
Therefore Jensen's inequality gives us that for $t\in [0,T_*]$ and $0<\eta\le k-2$,
\begin{align*}
\frac{2}{t}\int_{t/2}^{t} \rVert \partial_2u_2(s)\rVert_{L^\infty} ds&\le C_\eta \frac{2}{t}\int_{t/2}^{t}  \left(\rVert u_2(s)\rVert_{H^2}^2\right)^{\frac{k-2-\eta}{2(k-2)}}\left(\rVert u(s)\rVert_{H^{k}}^2\right)^{\frac{\eta}{2(k-2)}} ds\\
&\le C_\eta \left(\frac{2}{t}\int_{t/2}^t\rVert u_2(s)\rVert_{H^2}^2ds \right)^{\frac{k-2-\eta}{2(k-2)}}\left(\frac{2}{t}\int_{t/2}^t\rVert u(s)\rVert_{H^k}^2ds \right)^{\frac{\eta}{2(k-2)}}\\
 &\le C_\eta \left(\frac{M\delta}{t^k} \right)^{\frac{k-2-\eta}{2(k-2)}}\left( \frac{\delta}{t}\right)^{\frac{\eta}{2(k-2)}}\\
&\le C_\eta \frac{(M\delta)^{1/2}}{t^{\frac{k(k-2-\eta)}{2(k-2)} +\frac{\eta}{2(k-2)}}},
\end{align*}
where the second last inequality is due to the estimates for $\frac{2}{t}\int_{t/2}^{t} \rVert u_2(s)\rVert_{H^2}^2 ds$ and $\int_{t/2}^{t} \rVert u(s)\rVert_{H^k}^2 ds$ in \eqref{contradict3333ion3_12323} and \eqref{assumption_delta}. Since $k>2$, we can choose $\eta=\eta(k)\in (0,k-2)$ small enough, so that 
\[
\frac{k(k-2-\eta)}{2(k-2)} +\frac{\eta}{2(k-2)} >1.
\]
Hence, Lemma~\ref{Time_average_2} gives
\begin{align}\label{partial_2theta}
\int_0^{T_*} \rVert \partial_2u_2\rVert_{L^\infty}dt \le C{(M\delta)^{1/2}}.
\end{align}
Hence combining this with \eqref{assumption_delta}, we have
\begin{align}\label{h_integrability}
\int_0^{T_*}h(s)ds = \int_0^{T_*}C\rVert u(s)\rVert_{H^k}^2 +C\rVert \partial_2u_2(s)\rVert_{L^\infty}ds\le_C \delta + (M\delta)^{1/2}\le C(M\delta_0)^{1/2}\ll 1.
\end{align}
for sufficiently small $\delta_0$, depending on $C$ and $M$ (see \eqref{contradict3333ion3_1}).
Therefore, applying Lemma~\ref{diff_lem_time} to \eqref{laplacian_112sdxc2} with the estimates in \eqref{time_av_appl1} and \eqref{time_av_appl2},  we obtain
\begin{align}\label{fastest_decaying_term}
\frac{1}{t}\int_{t/2}^{t}f_1(s)ds  \le \frac{CM\delta}{t^{k+1}},\text{ for $t\in [0,T_*]$.}
\end{align}

 Next, we denote 
 \[
  H_2(t):=\rVert \partial_{12}\theta(t)\rVert_{L^2}^2,\quad f_2(t):=  C\rVert u_2\rVert_{{H}^2}^2,\quad g_2(t):=C\left(\sqrt{f_1(t)} \sqrt{H_2(t)}\rVert \theta\rVert_{H^k} + \rVert u(t)\rVert_{L^2}^2\right),
 \]
 so that  \eqref{laplacian_22} reads
\begin{align}\label{lLsd222sdxc}
\frac{d}{dt}H_2(t)\le -f_2(t) + h(t)H_2(t) +g_2(t),
\end{align}
From 
\[
H_2(s)=\rVert \partial_2\Delta \Psi\rVert_{L^2}^2=\rVert u_1\rVert_{H^2}^2\le_C \left(\rVert u\rVert_{L^2}^2\right)^{\frac{k-2}k}\left(\rVert u\rVert_{H^k}^2\right)^{2/k},
\]
we see that
\begin{align}\label{h_2estimate}
\frac{2}{t}\int_{t/2}^t H_2(s)ds&\le_C \frac{2}{t}\int_{t/2}^t \left(\rVert u(s)\rVert_{L^2}^{2}\right)^{\frac{k-2}{k}}\left( \rVert u (s)\rVert_{H^k}^2\right)^{2/k}ds\nonumber\\
&\le_C   \left(\frac{2}{t}\int_{t/2}^t \rVert u(s)\rVert_{L^2}^{2}ds\right)^{\frac{k-2}k}\left(\frac{2}{t}\int_{t/2}^t \rVert u(s)\rVert_{L^2}^{2}ds \right)^{2/k}\nonumber\\
&\le_C \left(\frac{\delta}{t^{k+1}}\right)^{\frac{k-2}k}\left(\frac{\delta}{t}\right)^{2/k}\nonumber\\
&\le \frac{C\delta}{t^{k-1}},
\end{align}
where the second last inequality is due to Proposition~\ref{energy_analysis} and \eqref{assumption_delta}. For $g_2$, it follows from the estimate for $\rVert \theta\rVert_{H^k}=\rVert \rho(t)-\rho_s\rVert_{H^k}$ in  \eqref{assumption_delta} that
\begin{align}\label{g2estimate_1}
\frac{2}{t}\int_{t/2}^t g_2(s)ds &\le_C\delta\frac{2}{t}\int_{t/2}^{t}\sqrt{f_1(s)}\sqrt{H_2(s)}ds + \frac{2}{t}\int_{t/2}^t \rVert u(s)\rVert_{L^2}^2 ds\nonumber\\
&\le_C \delta\left(\frac{2}{t}\int_{t/2}^{t}f_1(s)ds\right)^{1/2}\left(\frac{2}{t}\int_{t/2}^{t}H_2(s)ds\right)^{1/2} + \frac{2}{t}\int_{t/2}^t \rVert u(s)\rVert_{L^2}^2 ds\nonumber\\
&\le_C \delta \left(\frac{M\delta}{t^{k+1}}\right)^{1/2}\left(\frac{\delta}{t^{k-1}}\right)^{1/2}+\frac{\delta}{t^{k+1}}\nonumber\\
&\le_C \frac{M^{1/2}\delta^2}{t^k}+\frac{\delta}{t^{k+1}}\nonumber\\
&\le \frac{\delta M^{1/2}}{t^k},
\end{align}
where the third inequality follows from \eqref{fastest_decaying_term}, \eqref{h_2estimate} and \eqref{udecay2}.
Applying Lemma~\ref{diff_lem_time} to \eqref{lLsd222sdxc} with the estimates in \eqref{h_integrability}, \eqref{h_2estimate} and \eqref{g2estimate_1}, we obtain
\begin{align*}
\frac{2}{t}\int_{t/2}^t f_2(s)ds \le \frac{C\delta}{t^k}.
\end{align*}
We put this estimate together with \eqref{h_2estimate} and \eqref{contradict3333ion3_1}, we arrive at
\begin{align*}
\frac{M\delta}{T_*^{k-1}} \le \frac{1}{T_*}\int_{T_*/2}^{T_*} H_2(s) +sf_2(s)ds \le \frac{C\delta}{T_*^{k-1}}.
\end{align*}
Since $C$ is a constant that depends only on $\gamma$, the above inequality shows that if $M$ were chosen sufficiently large depending on $\gamma$, we arrive at a contradiction, which proves that \eqref{contradict3333ion3_1} cannot occur.
\end{proof}

\begin{corollary}\label{integrability_u2}
There exist $\delta_0>0$  such that if $\rho_0-\rho_s\in C^\infty_c(\Omega)$ and \eqref{assumption_delta} holds, then
\[
\int_{0}^T \rVert \nabla u_2(s)\rVert_{L^\infty}ds\le \sqrt{\delta}.
\]
\end{corollary}
\begin{proof}
As we argued in the proof of the previous proposition, the Sobolev embedding theorems give us
\begin{align*}
\frac{2}{t}\int_{t/2}^{t} \rVert \nabla u_2(s)\rVert_{L^\infty} ds&\le C_\eta \frac{2}{t}\int_{t/2}^{t}  \left(\rVert u_2(s)\rVert_{H^2}^2\right)^{\frac{k-2-\eta}{2(k-2)}}\left(\rVert u(s)\rVert_{H^{k}}^2\right)^{\frac{\eta}{2(k-2)}} ds\\
 & \le  C_\eta \left(\frac{2}{t}\int_{t/2}^t\rVert u_2(s)\rVert_{H^2}^2ds \right)^{\frac{k-2-\eta}{2(k-2)}}\left(\frac{2}{t}\int_{t/2}^t\rVert u(s)\rVert_{H^k}^2ds \right)^{\frac{\eta}{2(k-2)}}, 
\end{align*}
for $\eta\in (0,k-2)$ and $t\in [0,T]$. Using \eqref{assumption_delta} and Proposition~\ref{u_2high1}, we have
\[
\frac{2}{t}\int_{t/2}^{t} \rVert \nabla u_2(s)\rVert_{L^\infty} ds\le C_\eta\left(\frac{\delta}{t^k} \right)^{\frac{k-2-\eta}{2(k-2)}}\left(\frac{\delta}{t}\right)^{\frac{\eta}{2(k-2)}}\le C_\eta \sqrt{\delta}\frac{1}{t^{\frac{k(k-2-\eta) + \eta}{2(k-2)}}},\text{ for $t\in [0,T]$}
\]
Choosing $\eta$ sufficiently small enough that $\frac{k(k-2-\eta) + \eta}{2(k-2)}>1$, Lemma~\ref{Time_average_2} gives 
\[
\int_{1}^{T} \rVert \nabla u_2(s)\rVert_{L^\infty} ds\le C\sqrt{\delta}.
\]
Since it is assumed that $\sup_{t\in [0,1]}\rVert \theta(t)\rVert_{H^k}<\sqrt{\delta}$ in \eqref{assumption_delta}, we \textcolor{black}{reach} 
\[
\sup_{t\in[0,1]}\rVert \nabla u_2(t)\rVert_{L^\infty}\le_C \sup_{t\in [0,1]}\rVert u_2(t)\rVert_{H^k}\le_C \sup_{t\in [0,1]}\rVert \theta(t)\rVert_{H^k}\le C\sqrt{\delta}\textcolor{black}{,}
\]
\textcolor{black}{as desired.}
\end{proof}

\section{Proof of Theorem~\ref{main_theorem}}
\begin{proof}[\textcolor{black}{Proof of Theorem~\ref{main_theorem}}]
We first note that in order to prove Theorem~\ref{main_theorem}, we can assume,  without loss of generality, that $\rho_0-\rho_s\in C^\infty_c(\Omega)$. To see this, we use the density of $C^\infty_c$ in $H^k$, that is, for any $\rho_0$ such that $\rVert \rho_0-\rho_s\rVert_{H^k}\le \epsilon$ for sufficiently small $\epsilon>0$, we  find a sequence $\rho_{0,N}$ such that
\[
\rho_{0,N}-\rho_s\in C^\infty_c(\Omega),\ \rVert \rho_{0,N}-\rho_{s}\rVert_{H^k}\le C\epsilon,\text{ and } \lim_{N\to\infty}\rVert \rho_{0,N}-\rho_0\rVert_{H^k}\to 0.
\]
If Theorem~\ref{main_theorem} holds for the initial data $\rho_{0,N}$, the corresponding solutions, $\rho_N(t)$ satisfy 
\begin{align}\label{rksppsd1stab}
\rVert \rho_N(t)-\rho_s\rVert_{H^k}\le C\epsilon,\quad \rVert \rho_{N}(t)-\rho_{0,N}^*\rVert_{L^2}\le \frac{C\epsilon}{t^{k/2}} \text{ for all $t>0$.}
\end{align}
Therefore, we can find a convergent subsequence such that, which we still denote it by $\rho_N$, and a limit $\rho$ such that 
\[
\rho_N-\rho_s\rightharpoonup^* \rho-\rho_s \text{ in $L^\infty([0,\infty); H^k)$},\quad \lim_{N\to\infty}\rVert \rho_N(t)-\rho(t)\rVert_{L^2}=0,\text{ for every $ t\in [0,\infty)$.}
\]
Since the constants $C$ in \eqref{rksppsd1stab} do not depend on $N$, taking $N\to\infty$ and using Lemma~\ref{Stability_1}, we arrive at
\[
\rVert \rho(t)-\rho_s\rVert_{H^k}\le C\epsilon,\quad \rVert \rho(t)-\rho_0^*\rVert_{L^2}\le \frac{C\epsilon}{t^{k/2}},\text{ for $t>0$.}
\]
Thanks to the wellposedness result in Proposition~\ref{lwp_IPM}, we can conclude that $\rho$ is the unique solution to the IPM \textcolor{black}{equation} with the initial density $\rho_0$.
 
In the rest of the proof, we will assume $\rho_0-\rho_s\in C^\infty_c$. We claim that there exist $\varepsilon_0=\varepsilon_0(\gamma)>0$ and $M=M(\gamma)>0$ such that if
\begin{align*}
\rVert\rho_0-{\rho}_s\rVert_{H^k}=:\varepsilon\le \varepsilon_0,\end{align*}
 then  there exists a unique global solution $\theta(t)$ to \eqref{IPM_theta} and  it holds that
\begin{align}\label{exsproof1}
\rVert \rho(t)-\rho_s\rVert_{H^k}^2 +\int_0^t \rVert \nabla \Psi(t)\rVert_{H^k}^2 dt< M\varepsilon^2 \text{ for all $t>0$}.
\end{align}
Let us suppose the claim holds for the moment. Then Proposition~\ref{energy_analysis} is applicable, that is, the solution $\rho(t)$ satisfies
\begin{align}\label{estimate_ipm14}
E(t)\le \frac{C\varepsilon^2}{t^k} \text{ for all $t>0$}.
\end{align}
Then Proposition~\ref{propoos} immediately gives us the desired asymptotic convergence \eqref{convergence_1}.

In the rest of the note, we aim to prove \eqref{exsproof1}. Towards a contradiction, 
suppose that for any $\varepsilon_0>0$ and $M>0$, we can find $T>0$ such that \eqref{exsproof1} breaks down. Let $T^*>0$  be the first time that \eqref{exsproof1} breaks down, that is,
\begin{align}\label{reallyjoey}
\rVert \rho(T^*)-\rho_s\rVert_{H^k(\Omega)}^2 +\int_0^{T^*} \rVert \nabla \Psi(T^*)\rVert_{H^k(\Omega)}^2 dt= M\varepsilon^2\ll 1,
\end{align} for some $M\gg 1$, which will be chosen later together with $\varepsilon_0>0$. Let us denote
\[
f(t):= \rVert \rho(t)-\rho_s\rVert_{H^k}^2 =\rVert \theta\rVert_{H^k}^2,\quad g(t):=\rVert \nabla \Psi\rVert_{H^k}^2.
\]
Since $T^*$ is the first time that \eqref{reallyjoey} occurs, it holds that 
\begin{align}\label{estimate_ipm1}
 f(t) + \int_0^t g(s)ds \le CM\varepsilon^2 \text{ for $t\in [0,T^*]$},\quad f(0)\le \varepsilon^2.
\end{align}
Using the Sobolev embedding $H^k(\Omega)\hookrightarrow W^{1,\infty}(\Omega)$ for $k>2$, we have
\begin{align}\label{textbfs2}
\rVert u_2\rVert_{W^{1,\infty}}\le_C \rVert u_2\rVert_{H^k} \le_C \rVert\nabla \Psi\rVert_{H^k}\rVert \le_C \rVert \theta\rVert_{H^k}\le C f(t)^{1/2}.
\end{align}
 Assuming $M\varepsilon^2$ is sufficiently small so that we can apply Proposition~\ref{energy_IPM_estimate}, we read the estimate \eqref{eq:eshn} as
\begin{align}\label{RG}
\frac{d}{dt}f(t)\le -C\rVert \nabla\Psi(t)\rVert_{H^k}^2 + Cf(t),\text{ for $t\in [0,T^*]$.}
\end{align}
Using $f(0)\le \varepsilon^2$, we can immediately deduce from \eqref{RG}  that $f(t)\le C\varepsilon^2e^{Ct},$ for $t\in [0,T^*]$. Since 
\[
g(t)= \rVert \nabla \Psi\rVert_{H^k}= \rVert \nabla \partial_1\Delta^{-1}\theta\rVert_{H^k}\le \rVert \theta\rVert_{H^k}\le f(t),
\]  we can easily deduce that
\begin{align}\label{rjjsdpwdwd1sd}
f(t) +  \int_0^{t}g(s)ds\le C\varepsilon^2 e^{Ct} + \int_0^t C\varepsilon^2 e^{Cs}ds\le C\varepsilon^2e^{Ct} \text{ for $t\in [0,T^*]$}.
\end{align}
Thus, for \eqref{reallyjoey} to occur, we must have $M\varepsilon^2 \le C \varepsilon^2 e^{CT^*}$. This gives us a lower bound of $T^*$,
\[
T^*\ge C \log M.
\]
Let us pick 
\begin{align}\label{lmiddlet}
T^*_M:=\log\left(\log M\right)\gg 1.
\end{align} Without loss of generality, we can assume $M$ is sufficiently large to ensure that 
\[
T^*_M < T^*.
\]
Then, it follows from \eqref{rjjsdpwdwd1sd} that\begin{align}\label{middleestimagte}
f(T^*_M) + \int_0^{T^*_M} g(t)dt \le C\varepsilon^2 (\log M)^C.
\end{align}
 
  Now, we will estimate $f(t)$ more carefully for $t \in [ T^*_M,T^*]$. Firstly, we can assume $\varepsilon$ is sufficiently small so that \eqref{reallyjoey} with Proposition~\ref{energy_analysis} gives us
  \begin{align}\label{u_2decay_5}
  \frac{1}{t}\int_{t/2}^{t}s\rVert  u(s)\rVert_{L^2}^2 ds&\le C\frac{M\varepsilon^2}{t^{k}},\text{ for $t\in [0,T^*]$.}
  \end{align}
  Hence, using Lemma~\ref{Time_average_2}, we have
  \[
  \int_1^{T^*}s\rVert  u_2(s)\rVert_{L^2}^2ds \le CM\varepsilon^2,
  \]
  yielding that
  \begin{align}\label{rksssppp1}
  \int_{T_M^*}^{T^*}\rVert u_2(s)\rVert_{L^2}^2ds \le \frac{1}{T_M^*}  \int_{T_M^*}^{T^*}s\rVert   u_2(s)\rVert_{L^2}^2ds\le (\log\log M)^{-1}M\varepsilon^2.
  \end{align}
  Furthermore, 
\textcolor{black}{applying} Corollary~\ref{integrability_u2} with \eqref{estimate_ipm1}, we obtain
\begin{align}\label{rksssppp2}
\int_{0}^{T^*}\rVert \nabla u_2(s)\rVert_{L^\infty}ds \le_C \sqrt{M\varepsilon^2} \text{ for $t\in [0,T^*]$}.
\end{align}
Therefore integrating \eqref{eq:eshn} over $t\in [T_M^*,T^*]$, we arrive at
\begin{align*}
f(T^*)- f(T_M^*) + C\int_{T_M^*}^{T^*} g(t)dt &\le_C \sup_{t\in [0,T^*]}\rVert \theta\rVert_{H^k}^2\int_{T_M^*}^{T^*}\rVert \nabla u_2(s)\rVert_{L^{\infty}}ds + \int_{T_M^*}^{T^*}\rVert u_2(s)\rVert_{L^2}^2ds\\
&  \le_C \left(M\varepsilon^2\right)^{3/2} + (\log\log M)^{-1} M\varepsilon^2,
\end{align*}
where the last inequality follows from \eqref{rksssppp1}, \eqref{rksssppp2} and \eqref{estimate_ipm1}.
Hence
\begin{align*}
f(T^*) + \int_0^{T^*}g(t)dt &\le_C f(T^*_M)+\int_0^{T_M^*}g(t)dt + \left(M\varepsilon^2\right)^{3/2} + (\log\log M)^{-1} M\varepsilon^2\\
& \le_C  \varepsilon^2 (\log M)^C +\left(M\varepsilon^2\right)^{3/2} + (\log\log M)^{-1} M\varepsilon^2,
\end{align*}
where the last inequality \textcolor{black}{is due to} \eqref{middleestimagte}. From our assumption on $T^*$ in \eqref{reallyjoey}, we deduce that
\[
M\varepsilon^2\le_C \varepsilon^2 (\log M)^C +\left(M\varepsilon^2\right)^{3/2} + (\log\log M)^{-1} M\varepsilon^2.
\]
Since $\varepsilon\leq \varepsilon_0 \ll 1$, we further obtain
\begin{equation*}
    M \leq_C (\log M)^C + M^{3/2}\varepsilon_0 + (\log\log M)^{-1}M.
\end{equation*}
Hence, if we fix $M$ sufficiently large, depending only on $C$, we must have $M\le_C M^{3/2}\varepsilon_0$\textcolor{black}{. In other words,}
\begin{align}\label{Msdlast_unequalit}
1\le C \varepsilon_0 M^{1/2}.
\end{align}
Once  $M$ is fixed, we can choose $\varepsilon_0>0$ small enough, depending only on $C$, so that \eqref{Msdlast_unequalit} gives a contradiction. This shows that for such choice of $M$ and $\varepsilon_0>0$,  \eqref{exsproof1} must hold for all $t>0$.
\end{proof}

\begin{appendix}
\section{Convergence rate in the linearized equation}\label{Sharpness_1}
We consider the linearized equation of \eqref{IPM_theta} around the simplest steady state $\rho_s(x):=1-x_2$, namely,
\begin{align}\label{linear_p1s}
\begin{cases}
\theta_t = u_2,\\
 u=\nabla^{\perp}(-\Delta)^{-1}\partial_1\theta\\
 \lim_{|x|\to\infty}u(x)=0.
 \end{cases}
\end{align}

\begin{theorem}\label{prop:sharp}
  Let $k>2$ be fixed. For each $0<\varepsilon<1$, there exists  $\theta_0^{\varepsilon} \in H^{k}(\Omega)$ for which the solution to the linearized equation \eqref{linear_p1s} exhibits
  \begin{equation}\label{est:sharp}
      \|\theta^\varepsilon(t)\|_{L^2} \geq C t^{-\frac{k}{2}-\varepsilon}\quad\mbox{for any}\,\, t\geq 1, 
  \end{equation}
  for some universal constant $C>0$.
\end{theorem}
\begin{proof}
For each $0<\varepsilon<1$, define $\theta_0^\varepsilon\in H^{k}(\Omega)$ via its Fourier transform as
\begin{equation*}
\widehat{\theta_n^\varepsilon}(t=0,\xi)= \begin{cases}
    |\xi|^{-k-\frac{1}{2}-2\varepsilon} & \mbox{for $n=1$ and $|\xi|\geq 1$}, \\
    0 & \mbox{otherwise}.
\end{cases}
\end{equation*}
\textcolor{black}{We can explicitly solve \eqref{linear_p1s} on the Fourier side,} obtaining
\[
 \widehat{\theta_1^\varepsilon}(t,\xi)=  e^{-t/(1+\xi^2)}|\xi|^{-k-\frac{1}{2}-2\varepsilon}\mathbf{1}_{|\xi|\ge 1},\text{ thus } \theta^\varepsilon(t,x)= \left(\int_{|\xi|\ge 1} e^{-t/(1+\xi^2)}|\xi|^{-k-\frac{1}2 - 2\varepsilon}e^{i\xi x_2}d\xi\right) e^{i x_1}.
\]
Then, Plancherel's identity implies that
\begin{align*}
\|\theta^\varepsilon(t)\|_{L^2}^2 &= \int_{|\xi|\ge 1} e^{-\frac{2t}{1+|\xi|^2}}|\xi|^{-2k -1 -4\varepsilon}d\xi \\
& \geq \int_{|\xi|\geq 1}e^{-\frac{2t}{\xi^2}}|\xi|^{-2k-1-4\varepsilon}\,d\xi\\
&= \int_{0}^{t}e^{-2\eta} \left(\frac{t}{\eta}\right)^{-k-\frac{1}{2}-2\varepsilon} t^{\frac{1}{2}}\eta^{-\frac{3}{2}}\,d\eta  \\
       &\geq t^{-k-2\varepsilon}\textcolor{black}{\int_{0}^{1}e^{-2\eta}\eta^{k+1}\, d\eta} \\
       &= C t^{-k-2\varepsilon},
\end{align*} where $C$ does not depend on $\epsilon$. In the above estimate, we used $\frac{1}{1+\xi^2} \leq \frac{1}{\xi^2}$, change of variable $\eta=\frac{t}{\xi^2}$ with evenness of the integrand in $\xi$, and $t\geq 1$. This shows \eqref{est:sharp}.
\end{proof}
\end{appendix}

\bibliographystyle{abbrv}
\bibliography{references}

\end{document}